\documentclass[a4paper,12pt]{article}
\usepackage{amssymb,amsmath,amsthm,dsfont}

\usepackage{pdfpages}
\RequirePackage[round]{natbib}
\RequirePackage[colorlinks,citecolor=blue,urlcolor=blue]{hyperref}
\usepackage[utf8]{inputenc}
\usepackage[english]{babel}
\usepackage[T1]{fontenc}    % use 8-bit T1 fonts
\usepackage{hyperref}       % hyperlinks
\usepackage{url}            % simple URL typesetting
\usepackage{booktabs}       % professional-quality tables
\usepackage{nicefrac}       % compact symbols for 1/2, etc.
\usepackage{xcolor}         % colors
\usepackage{relsize}
\usepackage{subcaption}
\usepackage{algorithm, algpseudocode}

\usepackage{float}
\usepackage{enumerate}
\usepackage{caption}
\usepackage{amsmath,amsfonts,amssymb,euscript,mathrsfs}
\usepackage{graphics}
 \usepackage{setspace}
\usepackage{fullpage}
\usepackage{dsfont}

\usepackage{geometry}
\usepackage{graphicx}
\graphicspath{{Images/}}

\usepackage[export]{adjustbox}

\usepackage{float}

\usepackage[labelformat=simple]{subcaption}

\numberwithin{equation}{subsection} 

\newcommand{\N}{\mathbb{N}}
\newcommand{\Z}{\mathbb{Z}}
\newcommand{\R}{\mathbb{R}}
\newcommand{\X}{\mathcal{X}}

\renewcommand{\d}{\mathrm{d}}

\renewcommand{\P}{\mathbb{P}}
\newcommand{\E}{\mathbb{E}}

\newcommand{\simp}{\mathcal{S}}

\makeatletter
\let\amsmath@bigm\bigm

\renewcommand{\bigm}[1]{%
  \ifcsname fenced@\string#1\endcsname
    \expandafter\@firstoftwo
  \else
    \expandafter\@secondoftwo
  \fi
  {\expandafter\amsmath@bigm\csname fenced@\string#1\endcsname}%
  {\amsmath@bigm#1}%
}

\newcommand{\DeclareFence}[2]{\@namedef{fenced@\string#1}{#2}}
\makeatother

\DeclareFence{\mid}{|}

\newcommand{\vertiii}[1]{{\left\vert\kern-0.25ex\left\vert\kern-0.25ex\left\vert #1 
    \right\vert\kern-0.25ex\right\vert\kern-0.25ex\right\vert}}
    
\newcommand\mydots{\makebox[1em][c]{.\hfil.\hfil.}}

\newtheorem{lem}{Lemma}[section]

\newtheorem{prop}[lem]{Proposition}

\newtheorem{theo}[lem]{Theorem}

\newtheorem{cor}[lem]{Corollary}

\theoremstyle{definition}

\renewcommand\qedsymbol{$\blacksquare$}

\usepackage{xcolor}

\title{Time series models on compact spaces, with an application to dynamic modeling of relative abundance data in Ecology}

{\small
\author{Guillaume Franchi {\it guillaume.franchi@ensai.fr} \\
           Univ Rennes, Ensai, CNRS, CREST -- UMR 9194, F-35000 Rennes, France\\
        Lionel Truquet {\it lionel.truquet@ensai.fr} \\
             Univ Rennes, Ensai, CNRS, CREST -- UMR 9194, F-35000 Rennes, France
             }}
\date{}
\begin{document}

\maketitle

\begin{abstract}
 Motivated by the dynamic modeling of relative abundance data in ecology, we introduce a general approach to model stationary Markovian or non Markovian time series on (relatively) compact spaces such as a hypercube, the simplex or a sphere in the Euclidean space. Our approach is based on a general construction of infinite memory models, called chains with complete connections. 
The two main ingredients involved in our generic construction are a parametric family of probability distributions on the state space and a map from the state space to the parameter space. Our framework encompasses Markovian models, observation-driven models and more general infinite memory models.
Simple conditions ensuring the existence and uniqueness of a stationary and ergodic path are given. We then study in more details statistical inference in two time series models on the simplex, based on either a Dirichlet or a multivariate logistic-normal conditional distribution. Usefulness of our models to analyze abundance data in ecosystems is also  discussed.
\end{abstract}
\vspace*{1.0cm}

\footnoterule

\section{Introduction}\label{sec:1}

Modeling non-Gaussian autoregressive time series has attracted lots of effort in the time series literature. 
The need of specific models is often motivated by specific state spaces for the time series of interest due to some constraints on its values. This problem concerns for instance discrete-valued time series, with counts or categorical data.  See for instance 
\citet{ferland2006integer} or \citet{davis2016theory} for count time series and \citet{davis2021count} for a recent review of existing models. For categorical data, we defer the reader to \citet{FM}, \citet{TF} or \citet{truquet2020coupling}. 
A general construction of observation-driven models for discrete-valued time series has been investigated recently in \citet{armillotta2022observation}.
Other Gaussian type time series models have been developed for continuous but bounded state spaces. This concerns for instance data on the unit interval for modeling rates, proportions or correlation; as in \citet{cribari2023beta}  or \citet{gorgi2021beta} among other. Dynamics for other kind of structured data also concern time series on the sphere, as in \citet{zhu2024spherical}, or time series on the simplex studied for instance in \citet{grunwald1993time} or \citet{zheng2017dirichlet}.
Time series on the simplex and its use for analyzing continuous proportions in Ecology was one of the main motivation for this paper and we defer the reader to Section \ref{sec:3} for a more complete discussion on this problem.

A general routine to obtain time series models with a specific conditional distribution is to consider the so-called observation-driven models, following the terminology of \citet{cox1981statistical}.
In this case,  a time series model is constructed from a parametric family of probability distributions $\left\{\mu_{\alpha}: \alpha\in K\right\}$ by assuming that the conditional distribution of the process $Y_t\vert Y_{t-1},Y_{t-2},\ldots$ is given from this parametric family and time-varying parameters $\alpha=\alpha_t$ satisfying a stochastic recursion of the form 
\begin{equation}\label{recursi}
\alpha_t=g\left(\alpha_{t-1},Y_{t-1}\right),
\end{equation}
for a suitable mapping $g$ and then depends on past values of the series.
Many time series models introduced in the previous references are constructed from this idea. Relevant approaches for constructing such models are sometimes based on a generalization of ARMA processes, using appropriate link functions $g$ between the latent variable $\alpha_t$ and the observation $Y_t$. See \citet{zheng2015generalized} and the references therein for such kind of construction.

In this paper, we introduce a method for defining stationary observation-driven models on compact state spaces but directly with a so-called "infinite memory" property. Related to our previous notations, we consider latent variables of the form $\alpha_t=\alpha\left(Y_{t-1},Y_{t-2},\ldots\right)$ for suitable measurable mapping $\alpha$ defined on a Hilbert cube. See \citet{doukhan2008weakly} for a construction of real-valued autoregressive processes of this form.
In the discrete setting, these models are often called chains with complete connections. See \citet{bressaud1999decay} or \citet{chazottes2020optimal} and the references therein for an introduction to the theory of such stochastic processes which have been introduced in probability theory and studied mainly for their connections with Gibbs measures in statistical mechanics.
It turns out that these models contain Markov models as a special case but also some non-Markovian models for which the latent variable $\alpha_t$ satisfies a stochastic recursion of the form (\ref{recursi}). Indeed, setting $g_t=g\left(\cdot,Y_{t-1}\right)$, the assumptions made on $g$ often guaranty some kind of convergence for the backward iterations $g_t\circ g_{t-1}\circ\cdots\circ g_{t-n}$, thus leading to a representation in term of a chain with complete connections. Note that this convergence property, sometimes called invertibility, is natural because it is often needed for studying consistency of pseudo-maximum likelihood estimator. See for instance assumption $B3$ in \citet{douc2013ergodicity}. 

The link between chains with complete connections and many useful categorical time series models has been discussed in \citet{TF} and \citet{truquet2020coupling} but the extension to more general compact state spaces is new and has not been investigated for time series analysis. This is one of the important contribution of the paper. Note however that our infinite-memory approach requires some restrictions, one of the most important is that the latent variable $\alpha_t$ has to be bounded for applying our results. This is then suitable when the state space is compact and the latent variables depends continuously on transformed past observations $h(Y_{t-j})$ with a bounded mapping $h$. Even in this case, such an extension is challenging because contrarily to most of the existing results, one cannot use Markov chain theory in this context because the process $\alpha_t$ is no more Markovian. Moreover, working directly with chains with complete connections can also have some advantages. In particular, one can assume more persistence with respect to past observations in the expression of the latent variable $\alpha_t$, in contrast to existing results based on Markov chains theory which leads to geometrically ergodic models and then to an exponential decay of this dependence. See in particular our result given in Proposition \ref{ODC} below. 

It should be also pointed out that our approach avoids data transformations for modeling the time series, an approach which has some drawbacks. For instance, on the open unit interval $(0,1)$ (which closure is compact), one can apply a one-to-one transformation $Z_t=f(Y_t)$ with $f:(0,1)\rightarrow \R$ and model the new real-valued time series with a standard ARMA model. The logit or arcsine transformation is standard in this context. However, the prediction of the original variable $Y_t$ will be $f^{-1}\left(\E(Z_t\vert Y_{t-1},Y_{t-2},\ldots)\right)$ which is sub-optimal in general since it does not coincide with the optimal $\mathbb{L}^2-$prediction $\E\left(Y_t\vert Y_{t-1},Y_{t-2},\ldots\right)$ when $f$ is not linear. See \citet{douma2018analysing} for a discussion about the problem of data transformation in the analysis of continuous proportions in Ecology.

The paper is organized as follows. In Section \ref{sec:2}, we give some mathematical results concerning the construction of stationary paths for infinite memory dynamical systems defined with a one-point conditional distribution. Many examples of time series on the simplex, the hypercube or the unit sphere are presented in Section \ref{sec:3} as an illustration of the general theory. 
In Section \ref{sec:4}, we focus on some time series models on the simplex for which we discuss statistical inference procedures as well as interpretations of model parameters. Numerical experiments and a real data application are provided in Section \ref{sec:5} and Section \ref{sec:6} respectively. A conclusion is given in Section \ref{sec:7} and proofs of some of our results are given in Section \ref{sec:8}. Finally, additional proofs are provided in an online supplementary material.

\section{Construction of ergodic processes on a compact domain}\label{sec:2}

The aim of this section is to provide a generic approach to construct stationary time series models taking values on a compact space.
We will first state an abstract result, Theorem \ref{theo:2.1} below, which gives simple sufficient conditions on a transition kernel defining the conditional distribution of the model and from which existence and uniqueness of a stationary path is guaranteed. These processes can be Markovian or not. We then discuss how to construct these transition kernels from a parametric family of probability distributions on a compact space and how to get parsimonious models following this approach. 

\subsection{Chains with complete connections}

In what follows, we consider some <<one-point>> transition kernels, possibly depending on infinitely many past values. These transition kernels are used to define the conditional distribution of the outcome, given the past observations. The corresponding stochastic processes, called chains with complete connections, are widely studied in the literature, at least for finite-state spaces. See for instance \citet{bressaud1999decay}, \citet{truquet2020coupling} or \citet{chazottes2020optimal}. We then first generalize a well-known result for getting stationarity and ergodicity conditions for these models when the state space is continuous. Our extension, which is particularly adapted to the case of compact state spaces, is more or less a rewriting exercise of existing proofs. However,  Theorem \ref{theo:2.1} and its Corollary \ref{corexist} provide some guidelines for constructing general time series models on the simplex or other compact state spaces, and for sake of completeness, we provide a proof in the Appendix section.

Let us first introduce some basic notations that will be used throughout the paper.

In what follows, $X$ denotes an arbitrary compact subspace of an Euclidean space. 
Some examples are given by the hypercube $X=[0,1]^d$, $d\geq 1$, the unit sphere
\begin{equation}\label{sphere}
X=\mathcal{S}^2:=\left\{x=(x_1,x_2,x_3)\in \R^3: x_1^2+x_2^2+x_3^2=1\right\},
\end{equation}
or the simplex $\simp_{d-1}$ of $\R^d$ defined by 
	\begin{equation}\label{simplex}
\simp_{d-1} = \left\lbrace (y_1,\dots , y_d) \in \left(\R_+^*\right)^d \ \bigm\mid \ \sum_{i=1}^d y_i=1 \right\rbrace.
	\end{equation}

For any sequence $y=\left(y_t\right)_{t\in \Z}$ valued in a given set $X$, we will denote for all $t \in \Z$:

	\begin{equation}
	y_t^- := \left( y_t , y_{t-1},y_{t-2},\dots \right) =  (y_{t-j})_{j \in \N} \in X^{\N}.
	\end{equation}

Finally, given two sequences $y=\left( y_n \right)_{n \in \N}$ and $z=\left( z_n \right)_{n \in \N}$ valued in $X^{\N}$, we denote for all integer $m\geq 1$:

	\begin{equation}
	y \overset{m}{=} z  \iff \forall i \in \left\lbrace 0, \dots m-1\right\rbrace, \ y_i = z_i.
	\end{equation}

We assume that $X$ is endowed with its Borel $\sigma$-field $\X = \mathcal{B}(X)$ and we consider a reference measure $\xi$ on $\X$. 
There often exists a natural reference measure. For instance the Lebesgue measure $\xi=\lambda_d$ on the hypercube $[0,1]^d$, the Haar measure on the sphere or the measure 
\begin{equation}\label{doms}
\xi\left(dx_1,\ldots,d x_d\right)=\mathds{1}_{[0,1]}(x_d)\delta_{1-\sum_{i=1}^{d-1}x_i}(d x_d)\lambda_{d-1}\left(d x_1,\ldots, d x_{d-1}\right),
\end{equation}
on the simplex.

Consider also a probability kernel $P$ with source $\left( X^{\N} , \X^{\otimes \N} \right)$ and target $(X, \X)$:

	\begin{equation}
	\begin{array}[t]{rc}
P : &  X^{\N} \times  \X \longrightarrow [0;1] \\
 & (y,A) \longmapsto P(A \mid y)
\end{array}
	\end{equation}

such that hypotheses \textbf{A1} and \textbf{A2} below are satisfied.

\begin{enumerate}%[label=\textbf{A\arabic*}]
	\item For all $y\in X^{\N}$, we have
	\begin{equation}
		P(\cdot \mid y) = p(\cdot \mid y) \cdot \xi
	\end{equation}
	where $p(\cdot \mid y)$ is a non-negative measurable function defined on $X$.
	\item \label{assumptionA2} The application
	\begin{equation}
		\begin{array}{c}
			\left( X^{\N} \times X  , \X^{\otimes \N} \otimes \X \right) \longrightarrow \left(\R_+, \mathcal{B}(\R_+)\right)\\
			(y,x) \longmapsto p(x \mid y)
		\end{array}
	\end{equation}
	is measurable.
\end{enumerate}

\medskip
We remind that the distance in total variation $d_{TV}$ can be defined for two probability measures $\mu, \nu$ absolutely continuous with respect to reference measure $\lambda$ (i.e. $\mu=f \cdot \lambda$ and $\nu = g \cdot \lambda$, where $f$ and $g$ are non-negative measurable functions) by:
    \begin{align}
        d_{TV}(\mu ; \nu) & = \dfrac{1}{2} \int | f-g| \d \lambda \notag \\
         & = \underset{0 \leqslant h \leqslant 1}{\sup} \  \left( \int h \d \mu - \int h \d \nu \right).
    \end{align}

Let us now define for all $m \in \N$:

	\begin{equation}\label{bm}
	b_m := \sup \ \left\lbrace d_{TV} \left( P(\cdot \mid y) , P(\cdot \mid z ) \right) \ \bigm\mid \ y \overset{m}{=} z \right\rbrace.
	\end{equation}
When $m=0$, $b_0$ has to be understood as a supremum over all the possible pairs of past sequences $y,z\in X^{\N}$.

For sake of completeness, we provide a detailed proof of the following result in the Appendix section.
\begin{theo}\label{theo:2.1}

Assume that $b_0 <1$ and $\displaystyle{\sum_{m \in \N}} b_m < +\infty$, then there exists a strictly stationary $(Y_t)_{t\in \Z}$ valued in $X$ such that

	\begin{equation} \label{eq:2.1.10}
	\forall t \in \Z,\quad \forall G \in \X, \ \P (Y_t \in G \ \mid \ Y_{t-1}^- = y_{t-1}^-) = P(G \mid y_{t-1}^- ).
	\end{equation}
Moreover, the distribution of the process $(Y_t)_{t\in \Z}$ is unique. Additionally, the process $(Y_t)_{t\in\Z}$ is ergodic.

\end{theo}
\paragraph{Note.} For Markovian time series, i.e. 
$$P\left(G\vert y_{t-1}^{-}\right)=\widetilde{P}\left(G\vert y_{t-1},\ldots,y_{t-p}\right),$$
for some positive integer $p$ and a probability kernel $\widetilde{P}$ from $\left(X^p,\mathcal{B}(X)^{\otimes p}\right)$ to $\left(X,\mathcal{B}(X)\right)$, 
it is only necessary to check the condition $b_0<1$ to ensure the existence of a unique stationary solution. Indeed, we have $b_m=0$ for $m\geq p+1$
and the condition $\sum_{m\in \N}b_m<\infty$ is automatically satisfied.

\subsection{Construction of observation-driven models on a compact space}\label{sec:2.2}
Theorem \ref{theo:2.1} does not provide a generic method for constructing some transition kernels satisfying the required regularity assumptions. In what follows, we consider a construction from a parametric family of probability distributions on $X$. Dynamic models can then be obtained by allowing the finite number of parameters of the density to depend on past observations. 
We then follow a general routine to obtain time series models with a specific conditional distribution. 
This kind of model is often called observation-driven, following the terminology of \citet{cox1981statistical}. Many time series models are constructed from this idea.
See for instance \citet{doukhan2015inference} for Poisson autoregressive models with infinite memory, \citet{robinson2006pseudo} for ARCH($\infty$) models or \citet{TF} for autoregressive categorical time series models.

In what follows, we consider a Borel subset $K$ of $\R^k$ and $\left\lbrace \mu_{\alpha} \right\rbrace_{\alpha \in K}$ be a family of probability distributions on $(X,\X)$ such that:
    \begin{equation}
        \forall \alpha \in F, \ \mu_{\alpha} = f_{\alpha} \cdot \xi
    \end{equation}
    where $f_{\alpha}$ is a measurable function defined on $X$.

We define the probability kernel:
    \begin{equation}\label{condi}
        P: \begin{array}[t]{l}
            X^{\N} \times \X \longrightarrow [0;1] \\
            (y,G) \longmapsto P(G \mid y) := \mu_{\alpha(y)}(G)
        \end{array}
    \end{equation}
    where $\alpha : X^{\N} \longrightarrow K$ is a fixed measurable function.

In what follows, we consider an arbitrary norm $\vert\cdot\vert$ on $\R^k$.
The corresponding operator norm on the space $\mathcal{M}_k$ of square-matrices with real coefficients will be denoted in the same way.
Finally, we denoted by $\mathcal{P}_1(\X)$ the set of probability measures in $\left(X,\mathcal{X}\right)$.
The following assumptions will be needed.

\begin{description}
\item[H1]
For any  $\xi\left(\{x\in X: \inf_{\alpha\in K}f_{\alpha}(x)>0\}\right)>0$.
\item[H2]
The following mapping is Lipschitz continuous.
            \begin{equation}
                \begin{array}{c}
                    \alpha \longmapsto \mu_{\alpha}  \\
                     \left(K , \vert \cdot \vert\right) \longrightarrow  \left( \mathcal{P}_1(\X), d_{TV} \right)
                \end{array}
            \end{equation}
    \item[H3] $\sum_{m \geqslant 0} \sup \ \left\lbrace \vert \alpha (u) - \alpha (v) \vert \bigm\mid u \overset{m}{=}v \right\rbrace < + \infty$.
    \end{description}

The following result can be used to construct Markov or non-Markov time series models on $X$ that satisfy the required regularity conditions given in Theorem \ref{theo:2.1}. Note that the condition $b_0<1$ follows directly from
{\bf H1}, since we have the equality
$$d_{TV}\left(\mu_{\alpha},\mu_{\alpha'}\right)=1-\int \min\left(f_{\alpha}(x),f_{\alpha'}(x)\right)\xi(dx).$$
Assumptions {\bf H2-H3} automatically imply the condition $\sum_{m\in \N}b_m<\infty$ and the proof of the following corollary is straightforward.
\begin{cor}\label{corexist}
\begin{enumerate}
\item 
Suppose that Assumptions {\bf H1-H3} hold true. There then exists a strictly stationary process $(Y_t)_{t\in Z}$ valued in $X$ such that:
    \begin{equation}\label{condiprim}
        \forall t \in \Z, \ \P \left( Y_t \in G \ \mid \ Y_{t-1}^-=y_{t-1}^- \right) = P (G \mid y_{t-1}^-).
    \end{equation}
Moreover, the distribution of the process $(Y_t)_{t\in \Z}$ is unique and the process $(Y_t)_{t\in\Z}$ is ergodic.
\item
Suppose that $\alpha$ only depends on a finite number of variables, i.e. $\alpha(y)=\widetilde{\alpha}\left(y_0,\ldots,y_{p-1}\right)$ for $y\in X^{\N}$.
Then if Assumption {\bf H1} holds true, there exists a strictly stationary process $(Y_t)_{t\in Z}$ valued in $X$ such that:
    \begin{equation}
        \forall t \in \Z, \ \P \left( Y_t \in G \ \mid \ Y_{t-1}^-=y_{t-1}^- \right) = P (G \mid y_{t-1}^-).
    \end{equation}
 Moreover, the distribution of the process $(Y_t)_{t\in \Z}$ is unique and the process $(Y_t)_{t\in\Z}$ is ergodic.   
\end{enumerate}
\end{cor}

\paragraph{Note.} The second point of Corollary \ref{corexist} concerns Markov processes and only involves Assumption {\bf H1} which is a standard Doeblin type condition that is often used for getting ergodic properties of Markov chains models with compact state spaces.
%%%%

We now give a specific result for some parsimonious models of the previous type by allowing the random parameter $\alpha_t$ to depend on $Y_{t-1}$ and on its lagged value $\alpha_{t-1}$. Using parsimonious versions of this type is very classical in the time series literature. It has been used for instance by \citet{bollerslev1986generalized} for conditionally heteroscedastic time series, \citet{fokianos2009poisson} for autoregressive count time series or 
\citet{FM} for categorical time series. When the density $f_{\alpha}$ defined above corresponds to the Dirichlet distribution, \citet{zheng2017dirichlet} also considered models of this type. A discussion and a comparison between our models and that of \citet{zheng2017dirichlet} will be given in Section \ref{sec::3.1}.

For two Borel subsets $K,K'$ of $\R^k$, let $g:K\rightarrow K'$ be a one-to-one continuous mapping such that $g^{-1}$ is Lipschitz on any compact subset of $K'$. Let also $h:X\rightarrow \R^k$ be a measurable and bounded mapping. We are going to study the existence of a stationary solution $\left((Y_t,\alpha_t)\right)_{t\in\Z}$ for the system
\begin{equation}\label{OD}
g\left(\alpha_t\right)=F+D h\left(Y_{t-1}\right)+E g\left(\alpha_{t-1}\right),\quad
P\left(G\vert Y_{t-1}^{-}\right)=\mu_{\alpha_t}(G).
\end{equation}
where $F\in \R^k$, $D,E$ are square-matrices of size $k$ and  $(\alpha_t)_{t\in\Z}$ is a stationary process taking values in $K$.

In what follows, for any square matrix $B$ with real coefficients, we denote by $\rho(B)$ the spectral radius of $B$. Moreover, for any subset $A$ of $\R^k$, we denote its closure by $\overline{A}$. 
\begin{prop}\label{ODC}
Suppose that $\rho(E)<1$ and set for $y\in X^{\N}$,
$$\zeta(y)=(I-E)^{-1}F+\sum_{j\geq 1}E^{j-1}D h(y_{j-1}).$$
Suppose that $\overline{\mbox{Range }(\zeta)}\subset K'$ and that Assumptions {\bf H1-H2} hold true.

There then exists a stationary process $\left((Y_t,\alpha_t)\right)_{t\in \Z}$ solution of (\ref{OD}). The distribution of such a process is unique and the process is ergodic. Moreover, we have the representation
$$\alpha\left(Y_{t-1}^{-}\right):=\alpha_t=g^{-1}\circ \zeta\left(Y_{t-1}^{-}\right).$$
\end{prop}

\section{Examples of time series models}\label{sec:3}

\subsection{Time series models on the simplex}
Since the seminal work of \citet{aitchison1982statistical}, the study of compositional data has attracted a lot of efforts in the statistical literature. See in particular \citet{paw2011compositional} or \citet{pawlowsky2015modeling} for interesting historical and methodological overviews about the development of this field. Compositional data are vectors taking values in the simplex of $\R^d$, i.e. vectors with non-negative coordinates that sum to one. This kind of data is encountered in many fields such as the analysis of microbiome (\citet{gloor2007microbiome}), household expenditure surveys in economics (\citet{fry2000compositional}) or in ecology (\citet{jackson1997compositional}). The analysis of compositional data requires some care, due to the sum constraint on each vector. As discussed in \citet{aitchison1994principles}, standard multivariate techniques are then inadequate to model these data. Basically, there are too main approaches in the literature for dealing with such data.

The first one is to use some specific probability distributions on the simplex, such as the Dirichlet distribution or the multivariate logistic-normal distribution. See in particular \citet{aitchison1980logistic} for a discussion and a comparison between these two well-known examples. More general classes of probability distributions on the simplex can be found in \citet{aitchison1985general} or in \citet{rayens1994dependence} among others. See also \citet{mateu2021distributions} for a more recent reference.

The second one concerns data transformations. The most popular one is the log-ratio transformation for mapping data from the simplex to the Euclidean space with the same dimension. See \citet{aitchison1986statistical} for a detailed discussion about the advantages of this approach. In particular, assuming that the log-ratio of a composition follows a Gaussian distribution yields to logistic-normal distribution for the original composition. Interestingly, there also exists a notion of geometry, called the Aitchison geometry, which gives a structure of Euclidean space for the simplex. See for instance \citet{pawlowsky2015modeling}, chapter $3$, for an introduction.

Time series analysis of compositional data is much less investigated. A natural approach is to fit VARMA models to log-ratio transformations. See for instance \citet{brunsdon1998time}, \citet{barcelo2011compositional} or \citet{kynvclova2015modeling}. Though quite natural, applying VARMA models to transformed data is interesting to approximate conditional means of log-ratio. As explained in \citet{douma2018analysing} in the context of regression models for continuous proportions in Ecology, this approach can lead to a bias for forecasting original quantities and beta or Dirichlet regressions are discussed as an alternative. Alternatively, some references use probability distributions on the simplex for modeling the conditional distribution of the compositional time series. See for instance the state-space model developed in \citet{grunwald1993time}, the Dirichlet ARMA model in \citet{zheng2017dirichlet} or the spherical autoregressive process considered by \citet{zhu2024spherical} which can used for time series modeling on the simplex.  

In this paper, we only discuss the distributional approach for time series on the simplex following the corresponding part of the literature mentioned above. The problem of using or not the Aitchison geometry, which already occurs in the static case, will not be discussed here. Our aim is simply to show that one can model ecological data in this way, with model parameters than can be easily interpreted to assess the influence of past proportions on the next one. We defer the reader to Section \ref{sec:4} for details.

We keep the notations of Subsection \ref{sec:2.2}.

\subsubsection{Dirichlet model}\label{sec::3.1}
In this part for a vector of positive parameters $ \alpha=\left(\alpha_1,\ldots,\alpha_d\right)$, we denote by $f_{\alpha}$ the density of the Dirichlet distribution with respect to the dominating measure $\xi$ defined in (\ref{doms}), that is 
$$f_{\alpha}(x)=\frac{\prod_{i=1}^d\Gamma(\alpha_i)}{\Gamma(\sum_{i=1}^d \alpha_i)}\prod_{i=1}^d x_i^{\alpha_i-1},\quad x\in [0,1]^d,$$
where the gamma function is defined by 
$$\Gamma(z)=\int_{0}^{\infty}x^{z-1}\exp(-x)dx,\quad z>0.$$
Here, we are interested in model parameters 
$$\phi=\sum_{i=1}^d\alpha_i,\quad \lambda_i=\alpha_i/\phi,\quad 1\leq i\leq d.$$
The vector $\lambda$ is the mean of the distribution while $\phi$ is interpreted as a scale parameter.
Decomposing the parameters in this way is classical. See in particular \citet{douma2018analysing} for regression models with a conditional Dirichlet distribution and \citet{gorgi2021beta} for Beta autoregressive processes (which corresponds to the case $d=2$).
In what follows, we set for $y\in\mathcal{S}_{d-1}$, $h_1(y)=(y_1,\ldots,y_{d-1})'$ and $h_2(y)=-\sum_{i=1}^d y_i\log(y_i)$. 
We are interested here in defining stationary processes $(Y_t)_{t\in\Z}$ solutions of (\ref{condi})

$$\P\left(Y_t\in G\vert Y_{t-1}^{-}\right)=\mu_{\alpha_t}(G),$$
with $\alpha_t=\left(\lambda_t,\phi_t\right)$ and
\begin{eqnarray}\label{Dirichlet1}
\lambda_{i,t}&=&\frac{\exp\left(\mu_{i,t}\right)}{1+\sum_{j=1}^{d-1}\exp\left(\mu_{j,t}\right)},\quad 1\leq i\leq d-1\nonumber\\
\mu_t&=&\left(\mu_{1,t},\ldots,\mu_{d-1,t}\right)'=A_0+\sum_{k\geq 1}A_k h_1(Y_{t-k}),\nonumber\\
\phi_t&=&\exp\left(a_0+\sum_{k\geq 1}a_k h_2(Y_{t-k})\right),\\\nonumber
\end{eqnarray}
for some $A_0\in\R^{d-1}$, $A_k\in \mathcal{M}_{d-1}$, $a_0,a_k\in \R$ for any $k\geq 1$. 

\paragraph{Notes}
\begin{enumerate}
    \item Observe that the model defined above uses a reference coordinate, the coordinate $d$. Using a reference coordinate is unavoidable in our context since $Y_t$ and $\lambda_t$, which represents the conditional mean $\E\left(Y_t\vert Y_{t-1}^{-}\right)$, have components summing to one. Hence we adopt an approach which is similar to logistic/multinomial regression models.
    For autoregressive categorical time series, the same approach has been used by \citet{TF}. See also \citet{FM}. In the context of Ecology, $d$ will play the rule of a reference species.
    \item The Dirichlet distribution is a classical probability distribution on the simplex. It has often been criticized due the extremal independence it imposes on the compositions. See in particular \citet{aitchison1985general}. On the other hand, one nice property of this distribution is that its margins follow beta distributions. In the context of Ecology, it has been suggested in \citet{marquet2017proportional} that the relative abundance of a given species in large ecosystems are often compatible with a beta distribution.  
    
    \item
    Note that in the expression of $\lambda_t$, untransformed lag-values of $Y_t$ are included. This important point will enable us to interpret the model in term of relative risk ratios involving the conditional mean process. See Section \ref{interDir} for details.
    In contrast, the scale parameter $\phi_t$ is defined from the lag-values $h_2(Y_{t-j})$, $j\geq 1$, which represent the Shannon entropy of $Y_{t-j}$ which is seen here as a random probability distribution. This specification is related to applications in Ecology. For a species abundance $y\in \mathcal{S}_{d-1}$, $h_2(y)$ is a good indicator of the diversity in the ecosystem. The higher this quantity will be, the more diversity we will have in the ecosystem.
    \end{enumerate}
    
We have the following result.

\begin{prop}\label{PDirichlet1}
Suppose that 
$$\sum_{k\geq 1}k\left(\vert a_k\vert+\vert A_k\vert\right)<\infty$$
in (\ref{Dirichlet1}). Then there exists a stationary process $(Y_t)_{t\in\Z}$ solution of (\ref{Dirichlet1}). The distribution of such a process is unique
and the process is ergodic.    
\end{prop}    

We now give a specific result for observation driven models as a specific case of (\ref{Dirichlet1}).
The following model specification will be studied.
\begin{eqnarray}\label{Dirichlet2}
\lambda_{i,t}&=&\frac{\exp\left(\mu_{i,t}\right)}{1+\sum_{j=1}^{d-1}\exp\left(\mu_{j,t}\right)},\quad 1\leq i\leq d-1\nonumber\\
\mu_t&=&\left(\mu_{1,t},\ldots,\mu_{d-1,t}\right)'=C+A h_1(Y_{t-1})+B \mu_{t-1},\nonumber\\
\log \phi_t&=& c+a h_2(Y_{t-1})+ b\log \phi_{t-1},\\\nonumber
\end{eqnarray}
where $A,B\in \mathcal{M}_{d-1}$, $C\in \R^{d-1}$ and $(a,b,c)\in\R^3$.
Related to the general formulation given in Subsection \ref{sec:2.2}, 
we set $k=d$ and define the one-to-one mapping $g:\left(\R_+^{*}\right)^k\rightarrow \R^k$
by  
$$g(\alpha)=\left(\mbox{alr }(\lambda),\log(\phi)\right)',$$
where the additive log-ratio transformation, denoted $\mathrm{alr}$, is the bijective application defined by:
	\begin{equation}
	\mathrm{alr} : \begin{array}[t]{l}
	\simp_{d-1} \longrightarrow \R^{d-1} \\
	\lambda \longmapsto \left( \log \left( \dfrac{\lambda_1}{\lambda_d} \right),\dots ,\log \left( \dfrac{\lambda_{d-1}}{\lambda_d} \right) \right).
	\end{array}
	\end{equation}
We will restrict $g$ to a compact subset of the form $K_r=\left\{\alpha\in (\R_+^{*})^d: 1/r\leq \alpha_i\leq r,\quad 1\leq i\leq d\right\}$ for some $r>1$ that will be precised below and we set $K'_r=g(K_r)$.
 Note that $g(\alpha_t)=\left(\mu'_t,\log(\phi_t)\right)'$.
For the discussion given below we define $g_1(\alpha)=\mbox{alr }(\lambda)'$
and $g_2(\alpha)=\log(\phi)$.

We then get the following result about existence of stationary solutions.

\begin{prop}\label{PDirichlet2}
Suppose that $\max\left(\vert b\vert,\rho(B)\right)<1$. There then exists $r>1$ such that the restriction $g$ to $K=K_r$ is one-to-one continuous mapping from $K$ to $K'=g(K)$ and with a Lispchitz continuous inverse. Additionally there exists a stationary process $\left((Y_t,\alpha_t)\right)_{t\in\Z}$ solution of (\ref{Dirichlet2}). The distribution of such a process is unique and the process is ergodic. 
\end{prop}

\paragraph{Note.}

     Let us compare our model specification with that of \citet{zheng2017dirichlet}, a reference in which a similar dynamic is defined for the conditional mean process $(\lambda_t)_{t\in\Z}$. 
    In \citet{zheng2017dirichlet}, the authors are interested by a construction of an ARMA dynamic. The function $h_1$ is given by the alr function whereas our formulation only includes the $d-1$ first untransformed coordinates of $y$. Additionally, their function $g_1$ 
    is defined as the digamma function in order to get $\E\left(h_1(Y_t)\vert Y_{t-1}^{-}\right)=g_1(\alpha_t)$, which is required to define an ARMA type dynamics. In contrast, our function $g_1$ is simply given by the alr function. But our focus here is different since we stated our model to interpret the parameters with relative risks ratio in the spirit of the multinomial regression, as discussed in Subsection \ref{interDir}. The use of untransformed lag values of $Y_t$ and the alr function in the $g$ function is crucial for this. We then believe that the two models are complementary. The ARMA type Dirichlet model is interesting for a prediction purpose. On the other hand, the Dirichlet model we propose is more interesting for interpretability of fitted parameters which is the most important point to get for many time series in Ecology for which the sample size is relatively small.
    
    Note also that the probabilistic techniques used to define stationary models are rather different. In \citet{zheng2017dirichlet}, Markov chains techniques are used. In contrast, our construction of observation-driven models is a particular case of the general infinite memory model (\ref{Dirichlet1}). One advantage of our approach is to be able to define autoregressive processes with a stronger dependence with respect to past values. For instance, in (\ref{Dirichlet1}), one can take $A_j(\ell,\ell')$ of the form $L(\ell,\ell')\cdot j^{-\kappa(\ell,\ell')}$ with 
    $\kappa(\ell,\ell')>2$ for $1\leq \ell,\ell'\leq d-1$, whereas the observation-driven model (\ref{Dirichlet2}) coincides with (\ref{Dirichlet1}) for $A_j=B^{j-1}A$ which automatically means that the parameters are restricted to have a geometric decay. Note that the number of parameters is exactly the same for both models (matrices $(L,\kappa)$ vs matrices $(A,B)$).
    On another hand, the main drawback of our approach is that the link functions $h$ in (\ref{Dirichlet1}) have to be bounded and we cannot deal with the models given in \citet{zheng2017dirichlet}. The main reason is that chains with complete connections have nice properties when the transition kernels remain bounded from below or at least for which the coefficients $b_m$ in (\ref{bm}) can be controlled. The use of unbounded link functions $h$ is then problematic to get similar results. 
    See in particular \citet{chazottes2020optimal} for a discussion about the conditions ensuring existence and uniqueness of a stationary probability measure for chains with complete connections.

\subsubsection{Multivariate logistic-normal autoregressive model}

Another classical probability distribution on the simplex is the multivariate logistic-normal distribution. See \citet{aitchison1980logistic} for the advantages of this probability distribution with respect to the Dirichlet distribution. Here for $\alpha=\left(\mu,\Sigma\right)$ where $\mu\in \R^{d-1}$ and $\Sigma\in \mathcal{M}_{d-1}$ is a covariance matrix.
One can the consider $\alpha$ as an element of $\R^k$ with $k=\frac{(d-1)(d+2)}{2}$. 
\medskip

The multivariate logistic-normal distribution has the following density with respect to the measure $\xi$ defined in (\ref{doms}).
    \begin{equation}
        f_{\mu , \Sigma}(y) = \dfrac{1}{\sqrt{(2\pi)^{d-1}\cdot \det (\Sigma)}}\times \dfrac{1}{\prod_{i=1}^d y_i} \times \exp \left( -\dfrac{1}{2} <\Sigma^{-1} \cdot (\mathrm{alr}(y)-\mu) , (\mathrm{alr}(y)-\mu)> \right).
    \end{equation}
Note that $Y$ is distributed according to this probability if and only if 
$\mbox{alr }(Y)$ follows a multivariate Gaussian distribution with mean $\mu$ and covariance matrix $\Sigma$. 

We are interested here in defining stationary processes $(Y_t)_{t\in\Z}$ solutions of 
\begin{equation}\label{condi3}
\P\left(Y_t\in G\vert Y_{t-1}^{-}\right)=\mu_{\alpha_t}(G)=\int_G f_{\alpha_t}d\xi,\quad G\in \X
\end{equation}
with $\alpha_t=\left(\mu_t,\Sigma_t\right)$ defined by 
\begin{eqnarray}\label{logistic}
\mu_t&=&A_0+\sum_{k\geq 1}A_k h_1(Y_{t-k})\\\nonumber
\Sigma_t&=& \exp\left(\sum_{k\geq 1}a_k h_2(Y_{t-k})\right)\cdot V\\\nonumber
\end{eqnarray}
for some $A_0\in \R^{d-1}$, $A_k\in\mathcal{M}_{d-1}, a_k\in \R$ for any $k\geq 1$ and $V$ is a covariance matrix. With respect to model (\ref{Dirichlet1}), the intercept term $a_0$ is set to $0$ for identifiability issues. 

Setting for $t\in \Z$, $\varepsilon_t=\mbox{alr }(Y_t)-\mu_t$, which conditional distribution with respect to past values is Gaussian with mean $0$, one can also rewrite model as follows.
\begin{equation}\label{mieux}
Y_{i,t}=\frac{\exp\left(\mu_{i,t}+\varepsilon_{i,t}\right)}{1+\sum_{j=1}^{d-1}\exp\left(\mu_{j,t}+\varepsilon_{j,t}\right)},\quad 1\leq i\leq d-1.
\end{equation}

Let us now give the sufficient conditions ensuring existence of a stationary path for (\ref{condi3}), (\ref{logistic}).
\begin{prop}\label{plogistic}
Suppose that 
$$\sum_{k\geq 1}k\left(\vert a_k\vert+\vert A_k\vert\right)<\infty.$$
Then there exists a stationary process $(Y_t)_{t\in\Z}$ solution of (\ref{condi3}) and (\ref{logistic}). The distribution of such a process is unique and the process is ergodic.    
\end{prop}

We now give a specific result when the model specification is given by 
\begin{eqnarray}\label{logistic2}
\mu_t&=&C+A h_1(Y_{t-1})+B\mu_{t-1}\\\nonumber
\Sigma_t&=& \phi_t\cdot V\\\nonumber
\log\phi_t&=&a h_2(Y_{t-1})+b \log\phi_{t-1}.\\\nonumber
\end{eqnarray}
for some $C\in \R^{d-1}$, $A,B \in\mathcal{M}_{d-1}$, $(a,b)\in\R^2$ and $V$ is a covariance matrix.

\begin{prop}\label{plogistic2}
Suppose that $\max\left(\vert b\vert,\rho(B)\right)<1$. Then there exists a stationary process $(Y_t)_{t\in\Z}$ solution of (\ref{condi3}) and (\ref{logistic2}). The distribution of such a process is unique and the process is ergodic.    
\end{prop}    
    
\subsection{A time series model on $[0,1]^d$}

Proportional data with spatial correlations such as shrub covers measured at different locations are also common in Ecology.
See for instance \citet{eskelson2011estimating} or \citet{feng2015regression}.
Typically, these multivariate data take values in the unit cube $X=[0,1]^d$.
When $d=1$, we already defined in Section \ref{sec::3.1} some models using a beta
conditional distribution. When $d>1$, one can try to use multivariate versions of the beta distribution as a prototype for the conditional distribution of our model. One of the simplest example of a multivariate versions of the binomial distribution has been studied in \citet{olkin2003bivariate}. 
In this case for $\alpha=\left(\alpha_1,\ldots,\alpha_{d+1}\right)$, a vector of positive real number, $\mu_{\alpha}$ denotes the probability distribution of a random vector $\left(Z_i/(Z_i+Z_{d+1})\right)_{1\leq i\leq d}$ where $Z_1,\ldots,Z_{d+1}$ are independent random variables following a Gamma distribution with respective parameters $\alpha_1,\ldots,\alpha_{d+1}$. We refer to \citet{olkin2003bivariate} for a precise expression of the corresponding probability density. See also \citet{arnold2011flexible} for a more flexible version of the multivariate binomial distribution. For conciseness, we only consider the simplest one.

Examples of generalized linear dynamics can be 
\begin{equation}\label{infmem}
\alpha_t=\exp\left(A_0+\sum_{j=1}^{\infty}A_j Y_{t-j}\right)
\end{equation}
where $A_0$ is a vector of $\R^{d+1}$ and for $j\geq 1$, $A_j$ is a matrix of size $(d+1)\times d$. In (\ref{infmem}), the exponential function is applied coordinatewise.

A more parsimonious model can be obtained as in (\ref{OD}) with 
\begin{equation}\label{infmem2}
\log \alpha_t=F+D Y_{t-1}+E\log\alpha_{t-1},
\end{equation}
where $F$ is a vector of $\R^{d+1}$, $D$ and $E$ are matrices of respective sizes $(d+1)\times d$ and $(d+1)\times (d+1)$ and the log function is also applied coordinatewise.

\begin{prop}\label{cube}
Suppose that $\sum_{j\geq 1}j\Vert A_j\Vert<\infty$ in (\ref{infmem}) or $\rho(E)<1$ in (\ref{infmem2}).
There then exists a stationary process $\left((Y_t,\alpha_t)\right)_{t\in\Z}$ solution of (\ref{condi})  and (\ref{condiprim}) or (\ref{OD}) respectively. The distribution of such a process is unique and is ergodic.    
\end{prop}

\paragraph{Note.} A general alternative approach to construct multivariate distributions with beta marginals is to use copula. In this case $\mu_{\alpha}$ is given by the distribution of the random vector $\left(F_{\alpha_1}^{-1}(U_i)\right)_{1\leq i\leq d}$
where for $\alpha_i=(a_i,b_i)$, $F_{\alpha_i}$ denotes the cdf of the beta distribution with parameters $(a_i,b_i)$ and $(U_1,\ldots,U_d)$ is a probability distribution on $[0,1]^d$ with uniform marginals (i.e. a copula). This is approach is used for instance in \citet{eskelson2011estimating} for defined multivariate regresssion models with beta marginals. Showing that stationary models (\ref{infmem}) and (\ref{infmem2}) can be defined in this context requires to check Assumptions {\bf H1-H2} for the corresponding measure $\mu_{\alpha}$. This will dependent on the copula model and we will not provide a general result here.

\subsection{von Mises-Fisher autoregressive process on the sphere}

Here we consider $X=\mathcal{S}^2$, the unit sphere of $\R^3$. For $\kappa\geq 0$ and $v$ a unit vector in $\R^3$ (for the Euclidean norm),
let $f_{\kappa,v}$ be the Fisher-von Mises density w.t.r. to the Lebesgue measure $\xi$ on the sphere, i.e.
$$f_{\kappa,v}(x)=C(\kappa)\exp\left(\kappa <v,x>\right),\quad C(\kappa)=\frac{\kappa}{4\pi \sinh(\kappa)}.$$
Note that this distribution can be extended to the unit sphere $\mathcal{S}^n$ of $\R^{n+1}$ for any positive integer $n$. In this case, $<x,y>$ still denotes the canonical scalar product between two vectors $x,y\in\R^n$ and $C(\kappa)$ has a more complex expression depending on the modified Bessel function of the first kind. 

See \citet{dhillon2003modeling} for basic properties of von Mises-Fisher distributions. 

The total variation between two distributions with fixed $\kappa$  is bounded by 
$$\frac{C(\kappa)}{2}\int \left\vert \exp\left(\kappa<v,x>\right)- \exp\left(\kappa<v',x>\right)\right\vert \xi(dx)\leq \frac{C(\kappa)\kappa e}{2}S_2 \Vert v-v'\Vert,$$
where $S_2=4\pi$ is the surface of the sphere. It is then Lipschitz w.r.t. the mean direction $\mu$ and Assumption {\bf H2} is then satisfied. Since the density is uniformly bounded from below, {\bf H1} is also automatically satisfied.

In what follows, we denote by $\mathcal{R}_{z,\theta}$ the rotation of axis $z$ and angle $\theta$.
We construct a time series $(Y_t)_{t\in\Z}$ on the sphere as follows. We assume that the conditional density of $Y_t$ given $Y_{t-1},Y_{t-2},\ldots$ is given by $f_{\kappa,v_t}$ where
\begin{equation}\label{specif}
v_t=v\left(Y_{t-1},Y_{t-2},\ldots\right).
\end{equation}
A simple construction of a mapping $\mu$ can be obtained starting with a rotation axis $s$ on the sphere and setting 
$$v\left(Y_{t-1},Y_{t-2},\ldots\right)=g_{1,Y_{t-1}}\circ\cdots \circ g_{k, Y_{t-k}}(s),$$
with $g_{i,y}(u)=\mathcal{R}_{u,\theta_i}y$, that is we start with a rotating $Y_{t-k}$ around $s$ with angle $\theta_k$ and then rotate $Y_{t-k+1}$
around this new axis and so on. 
We consider the case of an infinite number of composition, 
\begin{equation}\label{infi}
v\left(Y_{t-1},Y_{t-2},\ldots\right)=\lim_{k\rightarrow \infty}g_{1,Y_{t-1}}\circ\cdots \circ g_{k, Y_{t-k}}(s).
\end{equation}
A natural observation-driven model is obtained if we assume that $v_t=g_{1,Y_{t-1}}(v_{t-1})$.
We have the following result

\begin{prop}\label{sphere2}
Suppose that $\sum_{k=1}^{\infty}8^{k/2}\sqrt{\prod_{i=1}^k\left(1-\cos(\theta_i)\right)}<\infty$. There then exists a process $(Y_t)_{t\in\Z}$, taking values on the unit sphere, such that 
$$\P\left(Y_t\in A\vert Y_{t-1},Y_{t-2},\ldots\right)=\int_A f_{\kappa,v_t}(y)\nu(dy).$$
Moreover, the distribution of such process is unique.    
\end{prop}

\paragraph{Notes}
\begin{enumerate}
\item
Theorem $1$ applies for instance if $\theta_i\rightarrow 0$ as $i\rightarrow \infty$ but it is not necessary. One can also build a model of the form (\ref{infi}) such that 
\begin{equation}\label{infi2}
v_t=g_{Y_{t-1}}\left(v_{t-1}\right),\quad t\in\Z,
\end{equation}
where $g_{y}(v)=\mathcal{R}_{v,\theta}(y)$, and $\cos(\theta)>7/8$. 
Indeed in this case, one can show that $\mu_t$ coincides with (\ref{infi}) with $g_{i,y}=g_y$ for all $i\geq 1$ and Proposition \ref{sphere} applies. 
\item 
There exist several contributions about time series analysis on the sphere. \citet{zhu2024spherical} is an interesting recent contribution which extends some previous attempts to define regression models on the sphere.
Their construction is based on the Fisher-Rao geometry and on the difference $L:=x_2\ominus x_1$ between two non-colinear vectors of the unit sphere $\mathcal{S}^n$ of $\R^{n+1}$. The square matrix $L$ is skew-symmetric and the exponential $\exp(L)$ is the rotation that rotates ounterclockwise $x_1$ to $x_2$ in the plane generated by the two vectors. \citet{zhu2024spherical} defined an AR process for $Y_t\ominus m$ where $m$ is the mean of $Y_t$. In our setup, on can use this formulation to define 
$$v_t\ominus s=\Gamma+\sum_{j=1}^{\infty}\theta_j \left(Y_{t-j}\ominus m\right),$$
where $\Gamma$ is a skew-symmetric matrix, $s$ and $m$ are two given vectors on the sphere and the $\theta_i$'s are real coefficients.
Equivalently $v_t=\mathcal{R}_t s$ with 
$$\mathcal{R}_t=\exp\left(\Gamma+\sum_{j\geq 1}\theta_j\left(Y_{t-j}\ominus m\right)\right).$$
It is possible to show that if $\sum_{j\geq 1}j \vert \theta_j\vert<\infty$, then we get a stationary solution for these recursions.
Note that this approach can be used for any $n-$dimensional unit sphere.

\end{enumerate}

\section{Estimation for models on the simplex and interpretation}\label{sec:4}

In this section, we provide some inference procedures for the parameters of the models, assuming that a realization of $Y_1,\ldots,Y_n$ is available. 
We will now denote by $\theta$ the vector of unknown parameters in our models.

\subsection{The Dirichlet model}

\subsection{Interpretation of parameters for the Dirichlet model }\label{interDir}

We now give an interpretation of the parameters for model (\ref{Dirichlet1}). 
We first recall that if $Y$ is a random vector distributed according to the Dirichlet distribution with scale parameter $\phi$ and mean parameters $\lambda_i$, $1\leq i\leq d$, then $\sum_{i=1}^d \mbox{Var }(Y_i)=\frac{1-\sum_{i=1}^d \lambda_i^2}{\phi+1}$.
Then an increase of $\phi$ leads to less individual variability for the proportions.
For the autoregressive model (\ref{Dirichlet1}), $h_2(Y_{t-k})$ is a positive quantity representing the Shannon entropy of the proportions $Y_{t-k}$.
In the context of Ecology, this non-negative quantity, which is minimal for a corner of the simplex and maximal for the uniform distribution, represents the disorder in the biodiversity at time $t-k$. See for instance \citet{chakrabarti2010maximum} or \citet{vranken2015review} and the references therein for an overview of the use of Shannon entropy to measure the biodiversity in ecosystems.
If a coefficient $a_k$ is positive, then the most diverse is the population at time $t-k$, the less variable will be the different proportions at time $t$. 

 We next give an interpretation of matrices $A_k$ in (\ref{Dirichlet1}). 
Consider the conditional means ratio at time $t$ between two species $i$ and species $j$ different from the reference species $d$:

\begin{equation}\label{MR}
MR\left(i,j,Y_{t-1}^{-}\right) := \dfrac{\lambda_{i,t}}{\lambda_{j,t}}= \exp\left(\sum_{k\geq 1}\left[A_k(i,\cdot)-A_k(j,\cdot)\right]h_1(Y_{t-k})\right),
\end{equation}
where $A_k(i,\cdot)$ and $A_k(j,\cdot)$ denote respectively the $i^{\mathrm{th}}$ and the $j^{\mathrm{th}}$ line of matrix $A_k$.
Then for a given species $\ell =1,\ldots,d-1$ and $k\geq 1$, the sign of $A_k(i,\ell)-A_k(j,\ell)$  indicates if the proportions $Y_{t-k,\ell}$ has a positive or negative impact on the log-ratio $\log MR\left(i,j,Y_{t-1}^{-}\right)$. 
However, if we want to get a precise measure of the impact of one proportion on the future values, one can introduce a notion of perturbation. The notion of perturbation we present here is different from the standard one on the simplex and is adapted to our model formulation. See \citet{aitchison2005role} for the standard multiplicative perturbation method on the simplex. 
Suppose now that the abundance is modified at time $t-1$. Set $Z_{t-1}=Y_{t-1}+\gamma$ for some $\gamma\in\R^d$ such that $\sum_{i=1}^d\gamma_i=0$.
Then $h_1(Z_{t-1})=h_1(Y_{t-1})+h_1(\gamma)$. 
The evolution in the means ratio is then given by
\begin{align}\label{EMR}
		EMR (i,j,\gamma)  & := \dfrac{MR\left(i,j,Z_{t-1},Y_{t-2}^{-}\right)}{MR\left(i,j,Y_{t-1}^{-}\right)} \notag \\
		 & = \exp \left(\left[A_1(i,\cdot)-A_1(j,\cdot)\right]h_1(\gamma)\right).
\end{align}
Note that if $j=d$, one can also define the EMR quantity by
$$EMR (i,d,\gamma)=\exp \left(\left[A_1(i,\cdot)\right]h_1(\gamma)\right).$$
Suppose now that for two species $i$ and $j$, the quantity $\left[A_1(i,\cdot)-A_1(j,\cdot)\right]h_1(\gamma)$ is positive (or $A_1(i,\cdot)h_1(\gamma)>0$ if $j=d$). Then $EMR(i,j,\gamma)>1$ and
the ratio between the two conditional means will be larger at time $t$ under the new configuration at time $t-1$. The new configuration of the abundance at time $t-1$ will then promote species $i$ compared to species $j$ at time $t$. On the contrary, a negative value of this last quantity would mean that the new configuration promotes species $j$ compared to species $i$.

The quantities EMR introduced above are similar to the relative risk ratio used in multinomial regression.
See \citet{hilbe2009logistic}, p. $385$ for details. The single difference is that the conditional probabilities are simply replaced with the conditional means of some variables on the simplex. 

Next, let us give a more specific perturbation that will be used in our real data application. Let us increase the proportion the species of reference $d$ by $p$ percent, at the expense of species $i$ and $j$, with the perturbation
	\begin{equation}
	\gamma=\left(0, \dots ,0, \underset{\underset{i}{\uparrow}}{-c \cdot p},0,\dots ,0,\underset{\underset{j}{\uparrow}}{(c-1)\cdot p},0,\dots, 0 ,\underset{\underset{d}{\uparrow}}{p}\right)'.
	\end{equation}
	
Here, $c \in [0;1]$ is an indicator on how the species are affected by the perturbation. A value $c=0$ means that only species $j$ is affected, while $c=\dfrac{1}{2}$ means that both species $i$ and $j$ are equally affected by the perturbation.
The evolution in the means ratio will be given by the sign of $c\left(A_1(i,j)+A_1(j,i)-A_1(i,i)-A_1(j,j)\right)+A_1(j,j)-A_1(i,j)$.

Note that with such a model, even if we modify the abundance of only two species, there are repercussions on all the species in the ecosystem, which seems quite realistic. Note that we have now precise quantities to determine if a given perturbation will more affect a species $i$ with respect to another species $j$.

\paragraph{Notes.}
\begin{enumerate}
\item
In model (\ref{Dirichlet2}), the matrix $A$ has exactly the same interpretation as the matrix $A_1$ in model (\ref{Dirichlet1}), since the perturbation $\gamma$ will not affect $\mu_{t-1}$.
\item
Mathematically, it is also possible to define a perturbation occurring at lag $t-k$ with $k\geq 2$. Setting $Z_{t-k}=Y_{t-k}+\gamma$, we then recover the same interpretation for the differences $A_k(i,\cdot)-A_k(j,\cdot)$, $1\leq i,j\leq d-1$, appearing in the corresponding EMR index which can still be interpreted as a sensitivity measure with respect to past proportions. However, interpreting the impact of this perturbation on future values can be spurious since a perturbation occurring at time $t-k$ will also impact the values of $Y_{t-j}$,  $1\leq j\leq k-1$. A solution to measure the impact of a perturbation at time $t-k$ on many future values 
is to simulate several paths of the perturbed and the non-perturbed model and to estimate for $\ell\geq 2$, the ratio of conditional means
$$\frac{\E\left[Y_{i,t-k+\ell}\vert Z_{t-k},Y_{t-k-1}^{-}\right]\cdot \E\left[Y_{j,t-k+\ell}\vert Y_{t-k}^{-}\right]}{\E\left[Y_{i,t-k+\ell}\vert Y_{t-k}^{-}\right]\cdot \E\left[Y_{j,t-k+\ell}\vert Z_{t-k},Y_{t-k-1}^{-}\right]}.$$
Note that for $\ell=1$, the previous ratio coincides with the EMR index defined above.
\end{enumerate}

\subsubsection{Statistical inference of the Dirichlet model}

In this part, we are interested in estimating the parameters $\theta=\left(C',\mbox{vec}(A)',\mbox{vec}(B)',c,a,b\right)'$ in model (\ref{Dirichlet2}) or 
$$\theta=\left(A_0',\mbox{Vec}(A_1)',\ldots,\mbox{Vec}(A_p)',a_0,\ldots,a_p\right)'$$
in model (\ref{Dirichlet1}) when a finite-order autoregression (i.e. $A_j=0$ for $j\geq p+1$) is assumed.
For another parameter vector $\widetilde{\theta}$, we will denote by 
by $\alpha_t(\widetilde{\theta}),\lambda_t(\widetilde{\theta}),\mu_t(\widetilde{\theta}),\phi_t(\widetilde{\theta})$ the latent processes at time $t$ computed by replacing $\theta$ by $\widetilde{\theta}$ in their expression.
For instance,
$$\mu_t(\widetilde{\theta})=\widetilde{C}+\widetilde{A} h_1(Y_{t-1})+\widetilde{B} \mu_{t-1}(\widetilde{\theta}),\quad t\in\Z.$$
in (\ref{Dirichlet2}).

The first natural idea is the conditional maximum likelihood estimator.
For any $t \in \Z$, the density of $Y_{t+1}$ conditionally to $Y_t^-$ is given by
$$f_{\alpha_t(\theta)}(y)=\dfrac{\Gamma (\phi_t(\theta))}{\Gamma(\alpha_{1,t}(\theta))\cdots \Gamma(\alpha_{d,t}(\theta))} \prod_{i=1}^{d} y_i^{\alpha_{i,t}(\theta)-1},$$
for $y\in\mathcal{S}_{d-1}$.
We then define the conditional maximum likelihood estimator by 
\begin{equation}\label{EMV}
\hat{\theta}_{ML}=\arg\max_{\widetilde{\theta}\in \Theta}\frac{1}{n}\sum_{t=1}^n \log f_{\alpha_t(\widetilde{\theta})}(Y_t).
\end{equation}
where $\Theta$ is a subset of $\R^{\ell}$ with $\ell$ being the number of unknown parameters in the model. In model (\ref{Dirichlet2}), it is necessary 
to initialize the values of $\alpha_0(\widetilde{\theta})$ (by $0$ for instance). Assuming that $A_j=0$ for $j\geq p+1$ in model (\ref{Dirichlet1}),
one can modify (\ref{EMV}) by starting the summation at time $t=p+1$ and it is not necessary to initialize the latent processes.

Getting consistency and asymptotic normality of $\hat{\theta}_{ML}$ is quite classical when $\Theta$ is a compact subset of $\R^{\ell}$ and such that $\rho(\widetilde{B})<1$ for all $\widetilde{\theta}\in \Theta$ in the case of model (\ref{Dirichlet2}). See in particular \citet{debaly2023multivariate}, a reference  which provides general assumptions for consistency and asymptotic normality of conditional pseudo-maximum likelihood estimators for multivariate autoregressive time series. 
An additional assumption is however necessary for identification in model (\ref{Dirichlet2}). For $\widetilde{\theta}\in \Theta$, the equalities 
		\begin{equation}\label{identif}
			\forall k \in \N^{*}, \widetilde{B}^k \cdot A = B^k \cdot A,
		\end{equation}
should imply the equality $\widetilde{B}=B$.
See Section $3.4$ in \citet{debaly2023multivariate} for the meaning of this condition.

We next present an alternative method which is suitable to first estimate the parameters in the conditional mean $\lambda_t(\theta)$. Let us first decompose the vector of parameters $\theta$ as $\theta=\left(\theta_1,\theta_2\right)$ where 
$$\theta_1=\left(A_0',\mbox{Vec}(A_1)',\ldots,\mbox{Vec}(A_p)'\right)$$
denotes the parameters of the conditional mean. Note that 
$\lambda_t(\theta)=\lambda_t(\theta_1)$.
We then define 
\begin{eqnarray}\label{convexity}
\hat{\theta}_1&=&\arg\min_{\widetilde{\theta}_1\in \Theta_1}\frac{1}{n}\sum_{t=1}^n m_t\left(\widetilde{\theta}_1\right)\\\nonumber
m_t\left(\widetilde{\theta}_1\right)&=&-\sum_{i=1}^dY_{i,t}\log\lambda_{i,t}\left(\widetilde{\theta}_1\right).
\end{eqnarray}

Note that (\ref{convexity}) corresponds to the (opposite of the) conditional likelihood estimator for multinomial categorical responses. For the Dirichlet model, this estimator is a simple minimum of contrast estimator. It has however some interesting properties given below.

\begin{prop}\label{prop:4.8}
\begin{enumerate}
\item
We have $\E m_0\left(\widetilde{\theta}_1\right)\geq \E m_0\left(\theta_1\right)$ and the equality holds true if and only if $\lambda_0\left(\widetilde{\theta}_1\right)=\lambda_0\left(\theta_1\right)$ a.s.
\item
Suppose that model (\ref{Dirichlet1}) is $p-$Markov (i.e. $A_j=0$ for $j\geq p+1$). Then the mapping 
$$\widetilde{\theta}_1\mapsto \frac{1}{n}\sum_{t=1}^n m_t\left(\widetilde{\theta}_1\right)$$
is convex.
\end{enumerate}
\end{prop}

\paragraph{Notes}
\begin{enumerate}
\item
The first point of Proposition \ref{prop:4.8} is helpful to get identification of the parameters in the conditional mean process. Since the conditional distribution $Y_t\vert Y_{t-1},Y_{t-2}\ldots$ is continuous, the equivalence 
$$\lambda_0\left(\widetilde{\theta}_1\right)=\lambda_0\left(\widetilde{\theta}_1\right)\Longleftrightarrow \widetilde{\theta}_1=\theta_1$$
can then be obtained directly in the Markov case and from condition (\ref{identif}) for model (\ref{Dirichlet2}.
\item
For the Markov model, convexity of this second objective function is particularly interesting numerically with respect to conditional likelihood. In our numerical experiments, we noticed that minimization algorithms for likelihood estimation are very sensitive to the choice of the initial values of the parameters to optimize.
Then this alternative procedure is interesting for two reasons. First, one be only interested in estimating $\theta_1$ for prediction. Secondly, one can use this second method to initialize the likelihood estimation procedure. This is the approach chosen in our numerical experiments.
\end{enumerate}

Whatever the method used, standard conditions (as explained above) guaranty the asymptotic normality of the estimator, i.e. 
$$\sqrt{n}\left(\hat{\theta}-\theta\right)\Longrightarrow \mathcal{N}_{\ell}\left(0,V_{\theta}\right),$$
with an usual sandwich form for the asymptotic variance
$$V_{\theta}=\E\left(\Ddot{m}_0(\theta)\right)^{-1}\mbox{Var }\left(\Dot{m}_0(\theta)\right)\E\left(\Ddot{m}_0(\theta)\right)^{-1},$$
with $m_0(\theta)=-\log f_{\alpha_0(\theta)}(Y_0)$ and $\Dot{m}_0$, $\Ddot{m}_0$ denote respectively the gradient and the Hessian matrix of the mapping $m_0$.
The same kind of asymptotic property holds true for the second minimum of contrast estimator defined in (\ref{convexity}).

\subsection{The multivariate logistic-normal autoregressive model}

 \subsubsection{Interpretation on model parameters for the multivariate logistic-normal autoregressive models}
 Interpretation of model parameters is quite similar with respect to the Dirichlet model but with a few variations.
First the sign of the $a_k'$s indicate if the conditional variance of $\log(Y_{i,t}/Y_{d,t})$ is positively or negatively affected by lag-values of the Shannon entropy. On the other hand, the mean value of these log-ratio will be positively or negatively impacted by the previous proportions depending on the signs of the entries of the matrices $A_k$, $k\geq 1$. On the other hand, 
one can still give a more precise interpretation of the first matrix $A_1$ as for the Dirichlet model. First, assume that $a_k=0$ for any $k\geq 1$. In this case the $\left(\varepsilon_{i,t}\right)_{1\leq i\leq d-1}'$s given in (\ref{mieux}) are i.i.d. $\mathcal{N}_{d-1}(0,V)$ distributed. For a perturbation $Z_{t-1}=Y_{t-1}+\gamma$, one can then define similarly to (\ref{MR}) and (\ref{EMR}), the ratio
$$R\left(i,j,Y_t^{-}\right)=\frac{Y_{i,t}}{Y_{j,t}}=\exp\left(\sum_{k\geq 1}[A_k(i,\cdot)-A_k(j,\cdot)]h_1(Y_{t-k})+\varepsilon_{i,t}-\varepsilon_{j,t}\right)$$
and the associated relative risk ratio by
\begin{equation}\label{RRR}
RRR\left(i,j,\gamma\right)=\frac{R\left(i,j,Y_t,Z_{t-1},Y_{t-1}^{-}\right)}{R\left(i,j,Y_t^{-}\right)}=\exp\left(\left[A_1(i,\cdot)-A_1(j,\cdot)\right]h_1(\gamma)\right).
\end{equation}
The main difference with respect to the Dirichlet model is that the ratio between conditional means is replaced with the ratio of the original variables, but the interpretation is similar.

When the coefficients $a_k$ are not necessarily equal to $0$, the previous interpretation is no more valid since the conditional variance of $\varepsilon_t$ is also affected by the perturbation. We then define an alternative risk measure by 
\begin{eqnarray}\label{altRRR}
\Delta LR (i,j,\gamma)&=&\E\left[\log \left( \dfrac{Y_{t,i}}{Y_{t,j}} \right) \bigm\mid Z_{t-1}, Y_{t-2}^{-}\right]- \E \left[\log \left( \dfrac{Y_{t,i}}{Y_{t,j}} \right) \bigm\mid Y_{t-1}^{-}\right]\\\nonumber
	 &=&\left[A_1(i,\cdot)-A_1(j,\cdot)\right]h_1(\gamma),
	 \end{eqnarray}
where $Z_{t-1}=Y_{t-1}+\gamma$. As for the EMR quantity defined for the Dirichlet model, a positive or negative $\Delta LR (i,j,\gamma)$ indicates that the species $i$ in the ecosystems will promoted or not with respect to the perturbation $\gamma$.

\subsubsection{Estimation for multivariate logistic-normal autoregressive model}

As in the previous subsection, we are here interested in estimating the parameters\\ $\theta=\left(C',\mbox{Vec }(A)',\mbox{Vec }(B)',a,b,\mbox{Vec }(V)\right)$ in model (\ref{logistic}) or 
$$\theta=\left(A_0',\mbox{Vec }(A_1)',\ldots,\mbox{Vec }(A_p)', a_1,\ldots,a_p,\mbox{Vec }(V)\right)$$
in model (\ref{logistic2}).

Likelihood estimation is here also relevant. Using the probability density of the multivariate logistic-normal distribution, one can show that the MLE can be defined by 
\begin{eqnarray}\label{QMLE}
\hat{\theta}^{MLE}&=&\arg\min_{\widetilde{\theta}\in \Theta}\frac{1}{n}\sum_{t=1}^n m_t\left(\widetilde{\theta}\right),\nonumber\\
m_t\left(\widetilde{\theta}\right)&=&\left(\mbox{alr}(Y_t)-\mu_t\left(\widetilde{\theta}\right)\right)'\Sigma_t\left(\widetilde{\theta}\right)^{-1}\left(\mbox{alr}(Y_t)-\mu_t\left(\widetilde{\theta}\right)\right)+\log \det\left(\Sigma_t\left(\widetilde{\theta}\right)\right).
\end{eqnarray}
Note that (\ref{QMLE}) is similar the Gaussian QMLE used for ARCH processes. 
See for instance \citet{francq2019garch} for a detailed presentation of this kind of quasi-maximum likelihood estimation for conditionally heteroscedastic autoregressive models.
Since the parameters in the model can be decomposed into $\theta=\left(\theta_1,\theta_2\right)$ where $\theta_1$
is the vector of parameters only involved in the latent process $\mu_t$, one can also define the following least-squares estimator for estimating parameter $\theta_1$.
\begin{equation}\label{OLS}
\hat{\theta}^{LS}=\arg\min_{\widetilde{\theta}_1\in \Theta_1}\frac{1}{n}\sum_{t=1}^n\left\vert \mbox{alr}(Y_t)-\mu_t\left(\widetilde{\theta}_1\right)\right\vert^2,
\end{equation}
where $\vert\cdot\vert$ denotes the Euclidean norm.
Estimator (\ref{OLS}) is justified by (\ref{mieux}), since $\mbox{alr}(Y_t)$ has conditional mean $\mu_t$. Then estimator (\ref{OLS}) can be used to initialize likelihood optimization.
\paragraph{Note.}
For the Markov model, one can show that $\hat{\theta}_1^{LS}=\hat{\theta}_1^{MLE}$, a classical property occurring when estimating VAR type models.   
Estimation methods for VAR models can be found in \citet{lutkepohl2005new}, Chapter $3$.
The estimation of $\theta_2$ can be done in a second step with
$$\hat{\theta}_2=\arg\min_{\widetilde{\theta}_2\in \Theta_2}\frac{1}{n}\sum_{t=1}^n m_t\left(\hat{\theta}_1^{LS},\widetilde{\theta}_2\right).$$

Note that the second vector of parameters $\theta_2$ contains the coefficients of a positive definite matrix $V$. To define a suitable parameter space $\Theta_2$, one can use Cholesky decomposition, i.e. $V=LL'$ with $L$
is a lower triangular matrix with positive diagonal elements.

\section{Numerical experiments}\label{sec:5}

Here, we give a short simulation study to assess the 
good finite-sample properties of the estimators for the Dirichlet model. To this end, we use a simplified setup
with only one lag and with parameters values that are close to the ones found in the real data experiments.
We fix $d=3$ and set $a_0=1.5$, $a_1=0.7$, $A_0=\begin{pmatrix}-1\\-2\end{pmatrix}$ and $A_1=\begin{pmatrix} 4&3\\3&5\end{pmatrix}$.

\begin{table}
	\centering
	\begin{tabular}{cccccccccc}
	\hline
	Sample size & Estimator &$a_0$&$a_1$&$A_0(1)$&$A_0(2)$&$A_1(1,1)$&$A_1(1,2)$&$A_1(2,1)$&$A(2,2)$ \\
	\hline
	$n=100$ & Likelihood & $0.28$ & $0.43$ & $0.84$& $1.04$&$0.94$&$1.02$&$1.13$&$1.23$\\
	$n=500$ & Likelihood & $0.12$ & $0.17$ & $0.3$&$0.31$&$0.35$&$0.39$& $0.34$&$0.4$\\
	$n=100$ & Likelihood & $-$ & $-$ & $0.78$&$1.03$&$0.83$&$0.9$&$1.08$&$1.17$\\
	$n=100$ & Convex & $-$ & $-$ & $1.69$&$1.94$&$1.72$&$1.87$&$1.98$&$2.16$\\
	$n=500$& Likelihood & $-$&$-$& $0.3$&$0.33$&$0.32$&$0.36$&$0.36$&$0.38$\\
	$n=500$& Convex& $-$&$-$& $0.51$& $0.61$&$0.57$&$0.68$&$0.66$&$0.77$\\
	\hline
	\end{tabular}
	\caption{RMSE for the two estimators computed with $1000$ realizations. For the four last lines, the dispersion parameter is set to $\phi_t=2$ and is not estimated and the convex criterion is used for initializing the likelihood.\label{RMSE}}
\end{table}

In Table \ref{RMSE}, the RMSE are given for the maximum likelihood estimators when $n=100$ and $n=500$.
For the two first lines, the parameters in the conditional mean were initialized with the minimum of the convex contrast (\ref{convexity}). 
To compare the performances of the likelihood estimator and the estimator given in (\ref{convexity}), we used a set up with the same number of parameters to be estimated.
In this case, latent dispersion parameter is set to $\phi_t=2$ and is assumed to be known, while the parameters $A_0$ and $A_1$ given above were still used for the conditional mean. The four last lines of Table \ref{RMSE} show that the RMSE of the maximum likelihood estimators is always smaller than that of the estimators minimizing the convex contrast, even for the sample size $n=100$. However, the second estimator is very useful to initialize the non-convex optimization procedure for computing the MLE.      

Estimator (\ref{convexity}) could also have an advantage
when using penalized methods such as LASSO for estimating a large number of parameters when $d>>1$. This is outside the scope of this paper but estimating high dimensional models of this type could be an interesting perspective for future research.

\section{Application on a real data set}\label{sec:6}

In this section, we apply our theoretical results on a real data set, which comes from the collection BioTIME (see \citet{biotime2018} for details). BioTIME registers a vast amount of time series relative to the abundance of species, which have been collected all around the world, through multiple environmental studies.

\medskip

The data set upon which we will work is about a population of alpine birds during 38 years \emph{(from 1964 to 2001)} at a specific location in Scandinavia (see \citet{svensson2006species} for details). Forty-seven species are present in this study, but for simplification purpose, we will focus on three particular species : Anthus pratensis, Calcarius lapponicus and Oenanthe oenanthe, which are always present in large number, from one year to another.

Even if it reduces largely the spectrum of species, we will consider an ecosystem only formed by these three species. One can find in Figure \ref{birds} the graphics of relative abundance for these species from the year $1964$ to the year $2001$.

\begin{figure}[H]
	\centering
	\includegraphics[scale=0.5]{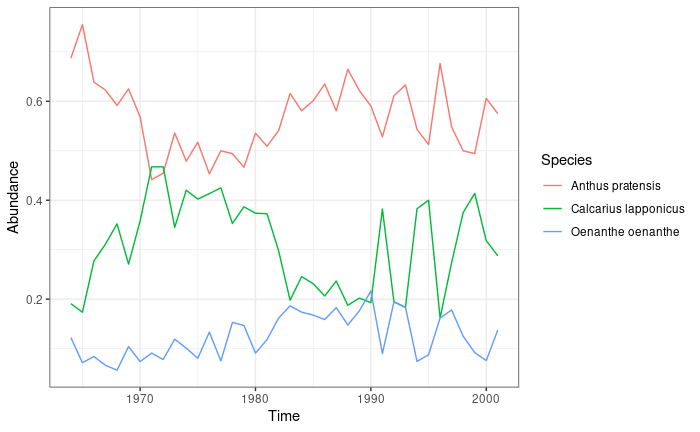}
	\caption{Abundance of Scandinavian Birds \label{birds}}
\end{figure}

%%%% MODIFICATION 23/01/23

We make here the assumption that the abundance of our three species of birds is an ergodic process $(Y_t)_{t \in \Z}$ valued in $\simp_{d-1}$ following the Dirichlet model defined in (\ref{Dirichlet1}) with $d=3$. The second logistic-normal model will only be presented for the interpretation of the parameters, since it led to a more extreme, yet similar, conclusion.
% We will not consider the second logistic-normal model which did not give a very different interpretation for this data set.

%%%% FIN MODIFICATION

The third species is used as the species of reference.
Due to the very short time series length $n=38$, we only use one lag to model the dynamic of this time series (i.e. $A_j$ and $a_j$ are set to $0$ for $j \geqslant 2$). The numbers of parameters in the model is $\ell=8$.

\subsection{Estimation of the parameters}\label{sec:6.1}

%%%% MODIFICATION 23/01/23

We begin by computing the minimum of the convex contrast (\ref{convexity}) and we only estimate $A_0$ and $A_1$. We found the estimates:
% \begin{equation}
% 	\hat{A}_1 = \left( \begin{array}{cc}
% 	3.98 & 3.15 \\
% 	2.81 & 4.70
% 	\end{array}\right) \quad \text{and} \quad
% 	\hat{A}_0 = \left( \begin{array}{c}
% 	-1.71 \\
% 	-2.13
% 	\end{array}\right).
% \end{equation}
\begin{equation}
	\hat{A}_1 = \left( \begin{array}{cc}
	2.73 & 1.22 \\
	0.84 & 4.00
	\end{array}\right) \quad \text{and} \quad
	\hat{A}_0 = \left( \begin{array}{c}
	-1.60 \\
	-2.27
	\end{array}\right).
\end{equation}
We next computed the maximum of conditional likelihood, using the previous estimates for initialization. We obtain the estimates:
	% \begin{equation}
	% 	\hat{A}^{MLE}_1 = \left( \begin{array}{cc}
	% 4.68 & 3.43 \\
	% 3.83 & 5.13
	% \end{array}\right), \quad
	% \hat{A}^{MLE}_0 = \left( \begin{array}{c}
	% -2.21 \\
	% -2.87
	% \end{array}\right) \quad \text{and} \quad
	% \begin{pmatrix} \hat{a}_0^{MLE}\\\hat{a}_1^{MLE}\end{pmatrix} = \left( \begin{array}{c}
	% 1.34 \\
	% 0.65
	% \end{array}\right).
	% \end{equation}
 	\begin{equation}
		\hat{A}^{MLE}_1 = \left( \begin{array}{cc}
	2.82 & 1.66 \\
	0.68 & 3.45
	\end{array}\right), \quad
	\hat{A}^{MLE}_0 = \left( \begin{array}{c}
	-1.70 \\
	-2.14
	\end{array}\right) \quad \text{and} \quad
	\begin{pmatrix} \hat{a}_0^{MLE}\\\hat{a}_1^{MLE}\end{pmatrix} = \left( \begin{array}{c}
	2.34 \\
	1.68
	\end{array}\right).
	\end{equation}
Standard errors for these estimates are not provided. They can easily be obtained from parametric bootstrap.
However, due to the very short length of the time series all the results requiring an asymptotic validity should be taken with caution.

\medskip

Let us now give an interpretation of these results in accordance with Subsection \ref{interDir}. The coefficient %$\hat{a}_1^{MLE}=0.65$%
$\hat{a}_1^{MLE}=1.68$ is positive and means that the more diversity there is in our ecosystem at some time $t$, the more stable the abundance will be at time $t+1$.
In order to interpret the coefficients in $\hat{A}_1^{MLE}$ we assume that the reference species $3$ increases its abundance by $p$ percents, at the expense of species $1$ and $2$:
	\begin{equation}
	z_t = y_t + (-c \cdot p, (c-1) \cdot p , p).
	\end{equation}
with $c \in [0;1]$ indicating on how the two first species are affected by the previous perturbation of the ecosystem.

We compute
	% \begin{equation}
	% \begin{array}{ll}
	% \dfrac{MR(1,2,z_t)}{MR(1,2,y_t)} & =\exp \left(p \left( c\left( \hat{A}_{12}+\hat{A}_{21}-\hat{A}_{11}-\hat{A}_{22}\right) \right) + \hat{A}_{22}-\hat{A}_{12}\right) \\
	%  & = \exp(p(-2.55 c +1.7))
	%  \end{array},
	% \end{equation}

 	\begin{equation} \label{eq:6.1.4}
	\begin{array}{ll}
	\log \left(\dfrac{MR(1,2,z_t)}{MR(1,2,y_t)}\right) & =p \left( c\left( \hat{A}_{12}+\hat{A}_{21}-\hat{A}_{11}-\hat{A}_{22}\right) \right) + \hat{A}_{22}-\hat{A}_{12}\\
	 & = p(-3.93 c +1.78)
	 \end{array}.
	\end{equation}

Let us recall that if this last quantity is greater than zero, then the perturbation is more beneficial to species 1 than it is to species 2.

\medskip

In Figure \ref{meanratio}, we draw the curve of (\ref{eq:6.1.4}) as a function of $c$ for different values of $p$.

\begin{figure}[H]
	\centering
        \begin{subfigure}[b]{0.45\textwidth}
	\includegraphics[scale=0.5]{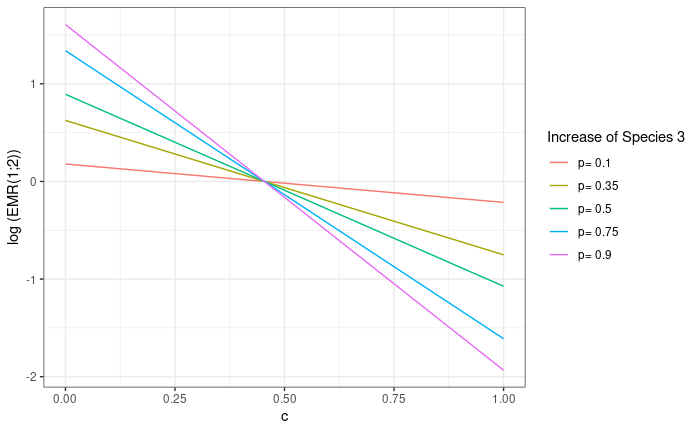}
	\caption{Evolution of the expected means ratio, Dirichlet model \label{meanratio}}
        \end{subfigure}
        \hfill
        \begin{subfigure}[b]{0.45\textwidth}
            \includegraphics[scale=0.5]{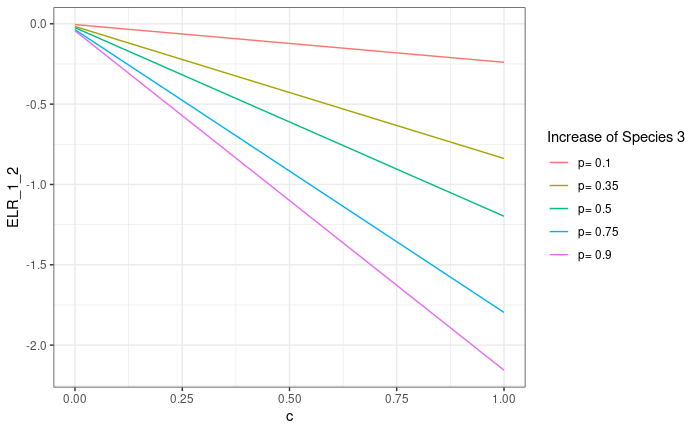}
	\caption{Evolution of the  expected log-ratio, logistic-normal model \label{logratio}}
        \end{subfigure}
    \caption{Evolution of the expected means ratio and log-ratio between species 1 and 2}
\end{figure}

When $c=0$, meaning that only the second species is affected by the perturbation, one can see that the new ecosystem will benefit to species 1 compared to species 2. It is the contrary if $c=1$.

\medskip

% Besides, it is interesting to consider what is happening when $c = \dfrac{1}{2}$. Indeed, the two first species are equally affected by the perturbation in this case. This scenario occurs for example if we introduce a certain number of individuals of species 3 in the ecosystem, without interfering with the other species.

% We can see that in this case, the evolution of the means ratio is greater than one, which means that the first species is advantaged by this new configuration compared to species 2. We can thus deduce that species 2 and 3 are more in competition in the ecosystem than species 1 and 3 are.

% \medskip

% Finally, for all four curves in Figure (\ref{meanratio}), there is no evolution in the means ratio for a value of $c$ greater than $\dfrac{1}{2}$, meaning that if we want to introduce some individuals of the last species without modifying the balance between the two other species, the first species should be more impacted by the perturbation than the second one. Therefore, we would have to remove some individuals of the first species to conserve the balance, because we introduced more competition for the second species.

It is interesting to take a look at the <<equilibrium point>>, i.e. when the ratio is equal to $0$. One can see on Figure \ref{meanratio} that this point is reached for a value of $c < \dfrac{1}{2}$, meaning that in order to conserve the same balance between species 1 and 2, it is necessary that species 1 should be less impacted by the perturbation. It therefore indicates that increasing the proportion of species 3 is more beneficial to species 2 than it is to species 1.

\medskip

A similar study with the logistic-normal model produced results in accordance with the Dirichlet model, even if the interpretation of the parameters is more radical. Indeed, for the logistic-normal model, we obtain a risk measure given by
    \begin{equation}
        \Delta LR (1,2,\gamma) = p \times (-2.35 c - 0.05),
    \end{equation}
which is negative for any $p \in [0,1]$ (see Figure \ref{logratio}). 

With this model, we would conclude that an increase of species 3 would always be beneficial to species 2 compared to species 1, regardless of how this increase would initially impact species 1 and 2. The Dirichlet model is more nuanced in the sense that if species 2 suffers too hard from the initial perturbation, species 1 will benefit from it.

\medskip

All of these interpretations should however be taken with caution, due to the shortness of our time series, and a larger data set would be necessary to confirm our findings.
%%%% FIN MODIFICATION

\subsection{Prediction}\label{sec:6.2}

We end this section by providing some guidelines for forecasting future abundances from the fitted Dirichlet model. Once again, our analysis is more illustrative for this small sample data set. 

To do so, we estimate the parameters of the model  with the maximum of conditional likelihood computed for the first thirty values \emph{(from 1964 to 1993)}.

\medskip

With the estimation $\hat{\theta}^{MLE}$ obtained, we perform our previsions on the last eight values \emph{(from 1994 to 2001)} and compare them to the true values. Since the model is parametric, it is simply necessary to simulate $b$ realizations of $\left( Z_t \right)_{31 \leqslant t \leqslant 38}$ using the Dirichlet model with the estimated parameters and with initial value $Z_{30}=Y_{30}$. We then compute the mean values of these realizations as well as 
the empirical quantiles of order $\alpha/2$ and $1-\alpha/2$ for $\alpha=5\%$. Of course, it is expected that the mean obtained in this way approximated the values of $\lambda_t(\hat{\theta})$ and then $\E\left(Y_t\vert Y_{30}\right)$ for $t=1,\ldots,8$.
The number of simulations is set to $b=10.000$.
The results are stored in Table \ref{prediction}.

\begin{table}[H]
	\centering
	\begin{tabular}{ccccc}
	\hline
	Year & Real values & Predicted values & Quantile 2.5\% & Quantile 97.5 \% \\
	\hline
	1994 & 0.543 & 0.617& 0.489 & 0.736\\
	1995 & 0.513 & 0.604 & 0.469 & 0.739\\
	1996 & 0.676 & 0.597 & 0.456 & 0.738\\
	1997 & 0.548 & 0.591 & 0.451 & 0.736\\
	1998 & 0.500 & 0.586 & 0.441 & 0.730\\
	1999 & 0.494 & 0.584 & 0.440 & 0.730\\
	2000 & 0.606 & 0.582 & 0.436 & 0.727\\
	2001 & 0.575 & 0.581 & 0.440 & 0.724\\
	\hline
	\end{tabular}
	\caption{Abundance of Anthus pratensis from 1994 to 2001 \label{prediction}}
\end{table}

\begin{table}[H]
	\centering
	\begin{tabular}{ccccc}
	\hline
	Year & Real values & Predicted values & Quantile 2.5\% & Quantile 97.5 \% \\
	\hline
	1994 & 0.383 & 0.219 & 0.125 & 0.333\\
	1995 & 0.400 & 0.242 & 0.128 & 0.376\\
	1996 & 0.162 & 0.256 & 0.132 & 0.399\\
	1997 & 0.274 & 0.265 & 0.136 & 0.414\\
	1998 & 0.375 & 0.271 & 0.138 & 0.425\\
	1999 & 0.414 & 0.276 & 0.141 & 0.430\\
	2000 & 0.318 & 0.279 & 0.143 & 0.439\\
	2001 & 0.288 & 0.281 & 0.146 & 0.442\\
	\hline
	\end{tabular}
	\caption{Abundance of Calcarius lapponicus from 1994 to 2001}
\end{table}

\begin{table}[H]
	\centering
	\begin{tabular}{ccccc}
	\hline
	Year & Real values & Predicted values & Quantile 2.5\% & Quantile 97.5 \% \\
	\hline
	1994 & 0.074 & 0.164 & 0.061 & 0.268\\
	1995 & 0.088 & 0.154 & 0.061 & 0.275\\
	1996 & 0.162 & 0.148 & 0.055 & 0.271\\
	1997 & 0.178 & 0.144 & 0.053 & 0.265\\
	1998 & 0.125 & 0.143 & 0.051 & 0.264\\
	1999 & 0.092 & 0.141 & 0.052 & 0.260\\
	2000 & 0.076 & 0.140 & 0.051 & 0.260\\
	2001 & 0.138 & 0.138 & 0.050 & 0.256\\
	\hline
	\end{tabular}
	\caption{Abundance of Oenanthe oenanthe from 1994 to 2001}
\end{table}
%%%% FIN MODIFICATION

In Figure \ref{Scandi} below, we plot in full line the observed values of abundance, and in dashed line the predicted values. The ribbons symbolize the confidence intervals.

% \begin{figure}[H]
% 	\centering
% 	\includegraphics[scale=0.5]{prevision.png}
% 	\caption{Prevision of the abundance for Scandinavian Birds \label{Scandi}}
% \end{figure}

\begin{figure}[H]
	\centering
	\includegraphics[scale=0.5]{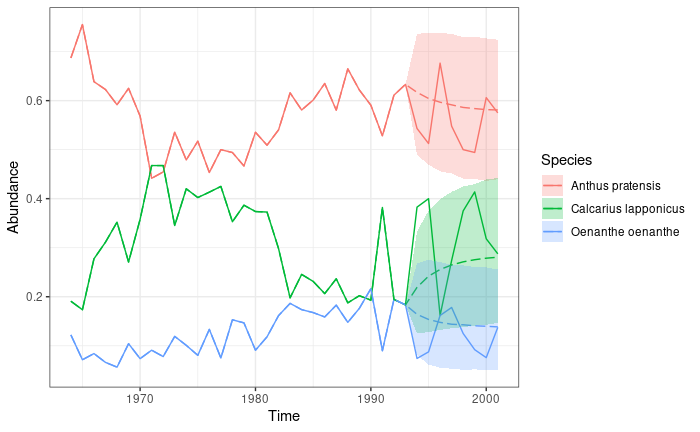}
	\caption{Prevision of the abundance for Scandinavian Birds, Dirichlet model \label{Scandi}}
\end{figure}

The first remark we can make about Figure \ref{Scandi} is that our previsions capture quite well the trend of the three abundances. Naturally, discrepancies occur due to the large variability of the observed time series.

Nevertheless, the observed time series stays in the confidence region, except for two consecutive of the species Calcarius lapponicus, showing that our model could apply well to this data set.

\section{Conclusion}\label{sec:7}

In this paper, we proposed a general probabilistic approach to construct stationary time series models on bounded state spaces. 
Many examples of time series models can be obtained in that way on the simplex, the sphere of the hypercube, with quite simple conditions to check on the model. As a specific case, we considered two models with a Dirichlet or a multivariate logistic-normal conditional distribution for modeling time series on the simplex. These two models have interesting interpretations, especially for studying the dynamic of relative abundance in Ecology. Efficient estimation procedures have also been presented. 

A further step for exploring this kind of models would be to consider additional probability distributions for which our assumptions are valid. For instance, on the simplex, the generalized Liouville distributions investigated in \citet{rayens1994dependence}
is one possibility since they offer a richer modeling than the simple Dirichlet distribution. However, interpretation of the parameters of these distributions could be more tricky. Copula-based models on the hypercube with beta marginals could be also interesting to investigate and there also exist other interesting probability distributions on the sphere. 

One of the challenging issues would be also to relax Assumption {\bf H3} and to consider latent processes $\alpha_t$ depending on unbounded functions of the response (such as the logit of lag-values for time series on the unit interval). This extension is generally much harder to obtain. See \citet{fernandez2005chains} for various general results which guaranty existence and uniqueness of stationary chains with complete connections. 

Finally, the use of (stationary) exogenous covariates in this kind of models is of practical interest but the probabilistic guarantees ensuring existence of stationary solutions are limited. \citet{truquet2020coupling} provided some guarantees for categorical time series.

\section{Proofs of the results}\label{sec:8}

\subsection{Proof of Proposition \ref{ODC}}

The result will be a consequence of Corollary \ref{corexist}. 
We first show that $\left((Y_t,\alpha_t)\right)_{t\in\Z}$ is a stationary solution of (\ref{OD}) if and only if $(Y_t)_{t\in \Z}$
is a stationary solution of 
\begin{equation}\label{equi}
\P\left(Y_t\in A\vert Y_{t-1}^{-}\right)=\mu_{\alpha(Y_{t-1}^-)}(A),
\end{equation}
with $\alpha(y)=g^{-1}\circ \zeta(y)$ for $y\in X^{\N}$.
Note that since $\rho(E)<1$ and $h$ is a bounded mapping, it is straightforward to show that $\overline{\mbox{Range }(\zeta)}$ is a compact subset of $K'$. Indeed, there exist $(C,\beta)\in \R_+\times (0,1)$ such that $\kappa_j:=\vert E^j\vert\leq C\beta^j$ for any $j\geq 1$ and $\mbox{Range }(\zeta)$ is a bounded subset of $\R^k$. 
Moreover, from the Lipschitz properties of $g^{-1}$, the mapping $\alpha$ defined above satisfies {\bf H3} and is is immediate to show that $b_m=O(\beta^m)$.

Suppose first that $\left((Y_t,\alpha_t)\right)_{t\in\Z}$ is a stationary solution of (\ref{OD}). We are going to show that $g(\alpha_t)$ is a bounded random variable. We have for any $j\geq 2$
\begin{equation}\label{precis}
g(\alpha_t)=E^j g(\alpha_{t-j})+\sum_{s=1}^{j-1}E^{s-1}\left(F+D h(Y_{t-s})\right).
\end{equation}
This leads to the bound $\left\vert g(\alpha_t)\right\vert\leq \kappa_j \left\vert g(\alpha_{t-j})\right\vert+v$, where 
$$v=\sum_{s\geq 1}\kappa_s\left(\vert F\vert+\vert D\vert\cdot \sup_{y\in X}\vert h(y)\vert\right).$$
From stationarity, we then obtain for $M>v$,
$$\P\left(\vert g(\alpha_0)\vert>M\right)\leq \P\left(\kappa_j\vert g(\alpha_0)\vert+v>M\right)$$
and the latter probability goes to $0$ if $j\rightarrow \infty$.
This shows that $g(\alpha_t)$ is bounded. From (\ref{precis}), by letting $j\rightarrow \infty$, we deduce that $g(\alpha_t)=\zeta\left(Y_{t-1}^{-}\right)$ a.s. and then $\alpha_t=g^{-1}\circ \zeta\left(Y_{t-1}^{-}\right)$ a.s. 
We conclude that the process $(Y_t)_{t\in\Z}$ is a stationary solution of (\ref{equi}). 

Conversely, suppose that $(Y_t)_{t\in\Z}$ is a stationary solution of (\ref{equi}). Setting $\alpha_t=\alpha(Y_{t-1}^{-})$, the process $\left((Y_t,\alpha_t)\right)_{t\in\Z}$ is clearly stationary. Moreover, it is not difficult to get the recursive equations
$$g(\alpha_t)=F+Dh(Y_{t-1})+Eg(\alpha_{t-1}),\quad t\in \Z.$$
We then conclude that the process $\left((Y_t,\alpha_t)\right)_{t\in\Z}$ is a stationary solution of (\ref{OD}).

From Corollary \ref{corexist}, we conclude that there exists a stationary solution to (\ref{OD}) and that the distribution of such process is unique. Since the solution of (\ref{equi}) is ergodic, we also deduce that the process $\left((Y_t,\alpha_t)\right)_{t\in\Z}$ is ergodic and the proof is complete.

\hfill \qedsymbol

\subsection{Proof of Proposition \ref{PDirichlet1}}

The result will follow from an application of Corollary \ref{corexist}. We have to check Assumptions {\bf H1-H3}. 
Let us define a compact $K$ of the form $[1/L,L]^{d+1}$ for some $L>1$ such that $\alpha(y)=(\lambda(y),\phi(y))\in K$ for any $y\in X^{\N}$ with 
$$\phi(y)=\exp\left(a_0+\sum_{k\geq 1}a_k h_2(y_k)\right),\quad \mbox{alr }(\lambda(y))=\mu(y),$$
$$\mu(y)=A_0+\sum_{k\geq 1} A_k y_k.$$
This is possible since the functions $h_1$ and $h_2$ are bounded and the functions $\exp$ and alr$^{-1}$ are upper and lower bounded on compact sets. Assumption {\bf H2} follows directly from Lemma \ref{DirichletLip} and Assumption {\bf H3} follows from our conditions on the coefficients and the fact that the functions $\exp$ and alr$^{-1}$ are Lipschitz on any compact sets. 

It remains to show {\bf H1}. We first note that for any compact interval $I\subset (0,\infty)$, the $\Gamma$ function satisfies $0<\inf_{z\in I}\Gamma(z)\leq \sup_{z\in I}\Gamma(z)<\infty$. Moreover $\prod_{i=1}^dx_i^{\alpha_i-1}\geq \prod_{i=1}^dx_i^{L-1}$ for $x\in [0,1]^d$.
It is now clear $\inf_{\alpha\in K}f_{\alpha}(x)>0$ except if $x_i=0$ for some $1\leq i\leq d$. This entails {\bf H1} and the proof is now complete.

\hfill \qedsymbol

\subsection{Proof of Proposition \ref{PDirichlet2}}

The proof will follow from Proposition \ref{ODC}. 
One can set $F=(C',c)'$, $D=\begin{pmatrix} A&0\\0&a\end{pmatrix}$ and $E=\begin{pmatrix} B& 0\\0&b\end{pmatrix}$. 
Then $\overline{\mbox{Range }(\zeta)}$ is a compact set $K''$ of $\R^d$. Because $g^{-1}$ is continuous, one can find $r>1$ such that $g^{-1}(K'')\subset K:=K_r$. It is then clear that the restriction of $g$ to $K$ is a one-to-one continuous mapping from $K$ to $g(K)$ and with a 
Lipschitz continuous inverse (recall that $g(K)$ is compact).
Moreover, the validity of {\bf H1-H2} has been shown in the proof of Proposition \ref{PDirichlet1}. The proof is now complete.$\square$

\subsection{Proof of Proposition \ref{plogistic}}

The result will follow from Corollary \ref{corexist}.
From our assumptions, the mapping $\alpha:=(\mu,\phi)$ such that 
$$\mu(y)=A_0+\sum_{k\geq 1}A_k h_1(Y_{t-k}),\quad \log\phi(y)=\sum_{k\geq 1}a_k h_2(y_k),$$
takes values in a compact set $K$ of the form $[-a,a]^{d-1}\times [1/b,b]$ for some $a>0$ and $b>1$. From Lemma \ref{LogisticLip}, {\bf H2} is satisfied.
Due to the fact that $\alpha$ takes values in the compact set $K$, {\bf H1} is easy to check. Finally {\bf H3} directly follows from our assumptions on the coefficients and the Lipschitz property of the exponential function on a compact interval.

\hfill \qedsymbol

\subsection{Proof of Proposition \ref{plogistic2}}

The result will follow from Proposition \ref{ODC}.
One can set $F=(C',c)'$, $D=\begin{pmatrix} A&0\\0&a\end{pmatrix}$ and $E=\begin{pmatrix} B& 0\\0&b\end{pmatrix}$. 
Here, we have $g(\alpha)=g(\mu,\phi)=(\mu,\log \phi)$.
 $\overline{\mbox{Range }(\zeta)}$ is contained in a compact set $K'$ 
 the form $[-a,a]^{d-1}\times [-c,c]$ where $a,c>0$ are large enough. 
 One can then set $K=[-a,a]^{d-1}\times [1/b,b]$ with $b=\exp(c)$.
 of $\R^d$. The validity of {\bf H1-H2} has been shown in the proof of Proposition \ref{plogistic} which completes the proof.

 \hfill \qedsymbol

\subsection{Proof of Theorem \ref{theo:2.1}}

The proof of this result is a quite straightforward extension of some results about chains with complete connections on finite set. See in particular \citet{bressaud1999decay}. We adapt here the different  steps of this extension, following \citet{truquet2020coupling}.

\medskip

We actually build a compatible family of finite-dimensional probability measures defined on the cylinders of $X^{\Z}$. Kolmogorov's extension theorem then guarantees the existence and the uniqueness of a probability measure on the infinite product $X^{\Z}$ and compatible with these finite-dimensional probability measures. Existence of a process $(Y_t)_{t\in \Z}$ with law $\pi$ easily follows and we will prove that it is stationary and ergodic.

\medskip

%%%% MODIFICATION 16/01/23

The first step of our proof relies on the maximal coupling applied recursively in order to construct two paths of the same process but with different initial conditions. %We defer the reader to \citet{den2012probability}, Theorem $2.12$ for definition of such a coupling. The proof of the following lemma is then straightforward.

Consider two probability measures $\mu$ and $\nu$ defined on $\mathcal{X}$ that are absolutely continuous with respect to a measure $\eta$:
    \begin{equation}
        \mu = f \cdot \eta \quad \text{and} \quad \nu = g \cdot \eta
    \end{equation}
where $f$ and $g$ are two non-negative functions defined on $X$.
A maximal coupling $\gamma$ of $\mu$ and $\nu$ (see \citet{den2012probability} for details) is given by
    \begin{equation} \label{eq:max_coup}
        \gamma (\d u , \d v) = f(u) \wedge g(u) \ \delta_u (\d v) \eta (\d u) + \dfrac{(f(u)-f(u)\wedge g(u))\times (g(v)-f(v)\wedge g(v))}{d_{TV}(\mu , \nu)} \ \eta (\d u) \eta (\d v).
    \end{equation}

\begin{lem} \label{lem:8.1}

Consider $x=(x_n)_{n\in \N}$ and $y=(y_n)_{n\in \N}$ two elements of $X^{\N}$.

There exists a process $(U_n^{x,y} ; V_n^{x,y})_{n\geqslant 1}$ valued in $\left( X^2 ; \X \otimes \X \right)$ such that for any integer $n\geq 1$ and $A\in \mathcal{B}(X)$, %$\ell\geq 0$
	\begin{gather}
%U_{-\ell}^{x,y}&=&x_{\ell},\quad V_{-\ell}^{x,y}=y_{\ell}\\
% \P\left(U_n^{x,y}\in A\vert U_{n-j}^{x,y}; j\geq 1\right)&=&P\left(A\vert (U_{n-j}^{x,y})_{j\geq 1}\right)\\
            \P\left(U_n^{x,y}\in A \mid U_{n-1}^{x,y}, \ldots , U_1^{x,y} \right)=P\left(A \mid U_{n-1}^{x,y}, \ldots , U_1^{x,y}\right) \label{eq:8.6.3} \\
 %\P\left(V_n^{x,y}\in A\vert V_{n-j}^{x,y}; j\geq 1\right)&=&P\left(A\vert (V_{n-j}^{x,y})_{j\geq 1}\right)\\
            \P\left(V_n^{x,y}\in A \mid V_{n-1}^{x,y}, \ldots , V_1^{x,y}\right)=P\left(A \mid V_{n-1}^{x,y}, \ldots , V_1^{x,y} \right) \label{eq:8.6.4} \\
            \P \left( U_n^{x,y} \neq V_n^{x,y} \mid (U_{n-j}^{x,y},V_{n-j}^{x,y})_{1 \leqslant j \leqslant n-1}\right)= d_{TV} \left( P\left(\cdot \mid (U_{n-j}^{x,y})_{1 \leqslant j \leqslant n-1}\right), P\left(\cdot \mid (V_{n-j}^{x,y})_{1 \leqslant j \leqslant n-1}\right)\right). \label{eq:8.6.5}
	\end{gather}

\end{lem}
    
\begin{proof}
    
    The main idea is to build by induction for $n \geqslant 1$ a probability measure $\mu_n$ defined on the measurable space $\left( \left((X^2\right)^n , \left(\mathcal{X}\otimes \mathcal{X}\right)^{\otimes n} \right)$ such that
        \begin{equation}
            \forall n \geqslant 1, \ \mu_{n+1} = K_n^{x,y} \cdot \mu_n,
        \end{equation}
    where $K_n^{x,y}$ is the transition kernel defined as follows.

    We set 
        \begin{equation}
            K_n^{x,y} : \begin{array}[t]{l}
                \mathcal{X}\otimes \mathcal{X} \times \left( X^2 \right)^n \longrightarrow [0 , 1]  \\
                C \times \left( (u_n,v_n),\mydots (u_1,v_1) \right) \longmapsto \gamma_{(u_n,v_n),\dots ,(u_1,v_1)}^{x,y} (C)
            \end{array}
        \end{equation}
        where $\gamma_{(u_n,v_n), \dots , (u_1,v_1)}^{x,y}$ is the maximal coupling of $P(\cdot \mid u_n,\ldots , u_1, x)$ and $P (\cdot \mid v_n, \ldots , v_1 , y)$ given by equation \ref{eq:max_coup}.
    
    Given a measurable set $C \in \mathcal{X}\otimes \mathcal{X}$, the measurability of $K_n^{x,y} (C , \cdot)$ relies on the measurability of function $p$ from assumption \ref{assumptionA2} and the measurability of integrals with a parameter.

    \medskip

    Furthermore, we initialize our sequence $(\mu_n)_{n \geqslant 1}$ with $\mu_1$ being the maximal coupling of $P(\cdot \mid x)$ and $P(\cdot \mid y)$ given by equation \ref{eq:max_coup}.

    \medskip
    
    Kolmogorov's consistency conditions are immediately satisfied for the sequence of measures $(\mu_n)_{n \geqslant 1}$, and the application of Kolmogorov's extension theorem leads to the desired result. Equations \ref{eq:8.6.3} to \ref{eq:8.6.5} are automatically satisfied by construction.
\end{proof}

    %%%% FIN MODIFICATION

We next define the Markov chain $(S_n)_{n \in \N}$ valued in $\N$ such that $S_0=0$ almost surely and for all $n$ and $i$ in $\N$:
	\begin{equation}
	\left\lbrace
	\begin{array}{l}
	\P (S_{n+1}=i+1 \ \mid \ S_n=i) = 1-b_i \\
	\P(S_{n+1}=0 \ \mid S_n = i ) = b_i.
	\end{array}
	\right.
	\end{equation}
    %%%% MODIFICATION 18/01/23
We denote for all $n \in \N$:
	\begin{equation}
            \begin{array}{l}
                b_0^* = b_0\\
                \forall n \geqslant 1, \ b_n^* = \P (S_n = 0).
            \end{array}
	\end{equation}
	
\begin{lem} \label{lem:8.2}
Let $x,y$ be two elements in $X^{\N}$, and define for all $n \in \N^*$:
	\begin{equation}
		T_n^{x,y} := \inf \ \left\lbrace m \in \left\lbrace 0,\ldots ,n-1 \right\rbrace \ \bigm\mid \ U_{n-m}^{x,y} \neq V_{n-m}^{x,y} \right\rbrace
	\end{equation}
with the convention
    \[\inf \emptyset = + \infty.\]

Then for all $n \in \N$:
	\begin{equation}
	\P \left( T_n^{x,y}=0 \right) = \P \left( U_n^{x,y} \neq V_n^{x,y} \right) \leqslant b_n^{*}.
	\end{equation}
\end{lem}

\begin{proof} In all that follows, we keep the same notations we used for the proof of Lemma \ref{lem:8.1}
Let us first remark that for $0 \leqslant k \leqslant n$, if $u_{n}=v_{n}, \dots , u_{n-k+1}=v_{n-k+1}$, we have by definition
	\begin{align}
	K_{n}^{x,y}(\Delta^c \mid (u_{n},v_{n}),\dots (u_1,v_1)) & = d_{TV}(P(\cdot \mid u_{n},\dots u_1,x) ; P(\cdot \mid v_{n},\dots , v_1,y)) \notag \\
 & \leqslant b_{k}.
	\end{align}
Thus, for any $n \in \N^*$ and $0 \leqslant k \leqslant n-1$
	\begin{align}
	\P(T_{n+1}^{x,y}=k+1 \mid T_n^{x,y}=k) & = \dfrac{\P(U_{n+1}^{x,y}=V_{n+1}^{x,y},\dots , U_{n+1-k}^{x,y}=V_{n+1-k}^{x,y},U_{n-k}^{x,y}\neq V_{n-k}^{x,y})}{\P(T_n^{x,y}=k)} \notag \\
 & = \dfrac{1}{\P(T_n^{x,y}=k)} \times \int \mathds{1}_{\Delta \times \dots \times \Delta \times \Delta^c}((u_n,v_n),\dots (u_{n-k+1},v_{n-k+1}),(u_{n-k},v_{n-k})) \notag \\
 & \hspace*{4cm} \times K_n^{x,y} ( \Delta \mid (u_n,v_n)\dots ,(u_1,v_1)) \mu_n ( d(u_n,v_n),\dots d(u_1,v_1)) \notag \\
 & \geqslant \dfrac{1}{\P(T_n^{x,y}=k)} \times (1-b_{k}) \times \P(U_n^{x,y}=V_n^{x,y},\dots , U_{n+1-k}^{x,y}=V_{n+1-k}^{x,y},U_{n-k}^{x,y}\neq V_{n-k}^{x,y}) \notag \\
 & =1-b_{k}
	\end{align}
and
	\begin{align}
	\P(T_{n+1}^{x,y}=0 \mid T_n^{x,y}=k) & \leqslant 1- \P(T_n^{x,y}=k+1 \mid T_n^{x,y}=k) \notag \\
& \leqslant b_{k}.
	\end{align}
We then show by induction over $n$:
	\begin{equation} \label{eq:8.6.15}
	\forall n \in \N^*, \ H_n : \ \forall k \in \N, \ \P(S_n \geqslant k) \leqslant \P(T_n^{x,y}\geqslant k).
	\end{equation}
\begin{enumerate}[label=$\triangleright$]
\item For $n=1$, we have for $k=0$
    \begin{equation}
        \P (S_1 \geqslant 0 ) = \P (T_1^{x,y}\geqslant 0) = 1.
    \end{equation}
For $k=1$, we have
    \begin{equation}
        \P (S_1 \geqslant 1) = \P (S_1=1) = 1-b_0
    \end{equation}
and
    \begin{equation}
        \P (T_1^{x,y}\geqslant 1) = \P (T_1^{x,y}=+\infty)=\P(U_1^{x,y}=V_1^{x,y})\geqslant 1-b_0
    \end{equation}
by construction of the process $\left( (U_n^{x,y},V_n^{x,y}) \right)_{n \geqslant 0}$.

Finally, for $k\geqslant 2$, we have $P(S_1 \geqslant k)=0 \leqslant \P (T_1^{x,y}\geqslant k)$, therefore proposition $H_0$ is true.
\item Assume $H_n$ is true for some $n \in \N$ fixed, let us show that $H_{n+1}$ is also true.

For $k=0$, the inequality is obvious, and for $k \geqslant n+2, \ \P(S_{n+1}\geqslant k)=0$, so the inequality is also trivial.

For $k=n+1$, we have
    \begin{equation}
        \P (S_{n+1} \geqslant n+1) = \P (S_{n+1}=n+1) = \prod_{j=0}^{n} (1-b_j)
    \end{equation}
and
    \begin{equation}
        \P (T_{n+1}^{x,y} \geqslant n+1) = \P (T_{n+1}^{x,y}=+\infty)=\P (U_{n+1}^{x,y}=V_{n+1}^{x,y},\ldots , U_1^{x,y}=V_1^{x,y})\geqslant \prod_{j=0}^k (1-b_j)
    \end{equation}
by construction of the process $\left( (U_n^{x,y},V_n^{x,y}) \right)_{n \geqslant 0}$, so the inequality is satisfied.

Take now $1 \leqslant k \leqslant n$, we have 
	\begin{align}
	\P(T_{n+1}^{x,y} \geqslant k) & = \sum_{m = k}^{n} \P(T_{n+1}^{x,y}=m) + \P(T_{n+1}^{x,y}=+\infty) \notag \\
	 & = \sum_{m = k}^n \P(T_{n+1}^{x,y}=m \mid T_n^{x,y}=m-1) \times \P(T_n^{x,y}=m-1) + \P (T_{n+1}^{x,y}=+\infty) \notag \\
	 & \geqslant \sum_{m = k}^n (1-b_{m-1}) \P(T_n^{x,y}=m-1) + \P (T_{n+1}^{x,y}=+\infty) \notag \\
	 & = \sum_{m = k}^{n} (1-b_{m-1})\left( \P (T_n^{x,y} \geqslant m-1)-\P(T_n^{x,y} \geqslant m) \right) + \P (T_{n+1}^{x,y}=+\infty) \notag \\
	 & = \sum_{m =k}^n \P ( T_n^{x,y} \geqslant m-1) - \sum_{m = k}^n b_{m-1} \P (T_n^{x,y} \geqslant m-1 ) \notag \\
	 & \hspace*{2cm} - \sum_{m = k}^n \P( T_n^{x,y} \geqslant m) + \sum_{ m = k}^n b_{m-1} \P (T_n^{x,y} \geqslant m ) + \P (T_{n+1}^{x,y} = + \infty) \notag \\
      & = (1-b_{k-1}) \P (T_n^{x,y} \geqslant k-1) + \sum_{m =k}^{n-1} (b_{m-1}-b_m) \P (T_n^{x,y} \geqslant m) + b_{n-1} \P (T_n^{x,y} \geqslant n) \notag \\
      & \hspace*{2cm}- \P (T_n^{x,y} \geqslant n)+\P (T_{n+1}^{x,y} =+ \infty) \notag \\
      & = (1-b_{k-1}) \P (T_n^{x,y} \geqslant k-1) + \sum_{m =k}^{n-1} (b_{m-1}-b_m) \P (T_n^{x,y} \geqslant m) + b_{n-1} \P (T_n^{x,y} \geqslant n) \notag \\
       & \hspace*{2cm} +\prod_{j=0}^{n-1} (1-b_j)(1-b_n-1) \notag \\
      & = (1-b_{k-1}) \P (T_n^{x,y} \geqslant k-1) + \sum_{m =k}^{n} (b_{m-1}-b_m) \P (T_n^{x,y} \geqslant m)
    \end{align}
Similarly
	\begin{align}
	\P (S_{n+1} \geqslant k) & = (1-b_{k-1}) \P (S_n \geqslant k-1) + \sum_{m = k}^n (b_{m-1}-b_m) \P (S_n \geqslant m) \notag \\
 	& \leqslant (1-b_{k-1}) \P (T_n^{x,y} \geqslant k-1) + \sum_{m = k}^n (b_{m-1}-b_m) \P (T_n^{x,,y} \geqslant m) \notag \\
 	& \hspace*{3cm} \text{since $H_n$ is true and $(b_m)_{m\geqslant 0}$ is decreasing} \notag \\
 	&  \leqslant \P (T_{n+1}^{x,y} \geqslant k),
	\end{align}
which concludes the induction.
\end{enumerate}
%%%% FIN MODIFICATION
Applying inequality \ref{eq:8.6.15} for $k=1$, we obtain for all $n \in \N$:
	\begin{equation}
	\P (S_n \geqslant 1) \leqslant \P (T_n^{x,y} \geqslant 1)
	\end{equation}
Thus
	\begin{equation}
	\P ( T_n^{x,y} =0) \leqslant b_n^*,
	\end{equation}
i.e. 
	\begin{equation}
	\P (U_n^{x,y}\neq V_n^{x,y}) \leqslant b_n^*.
	\end{equation}
\end{proof}
%%%%% MODIFICATION 19 / 01 / 23
Given $k \in \Z$, $l \in \N$ and $n \in \N$ such that $n+k \geqslant 1$, as well as a past $y=(y_0,y_{1},\dots ) \in X^{\N}$, we define the probability measure $\pi_{k,l}^{(n,y)}$ on $\left( X^{l+1}, \X^{\otimes l+1}\right)$ for all $A_k,\dots ,A_{k+l} \in \X$ by
	\begin{equation}
	\pi_{k,l}^{(n,y)}(A_k \times \dots \times A_{k+l}) = \underbrace{\int_{\X} \dots \int_{\X}}_{n+k-1}  \underbrace{\int_{A_k} \dots \int_{A_{k+l}}}_{l+1} P( \d z_{n+k+l} \mid z_{n+k+l-1},\dots ,z_{1},y ) \dots P(\d z_{1} \mid y).
	\end{equation}
Note that if $(Z_n^y)_{n\geqslant 1}$ is a process satisfying
	\begin{equation} \label{eq:8.6.27}
	\P (Z_1^{y} \in A_1) = P ( A_1 \mid y),
	\end{equation}
and
	\begin{equation}
	\forall n \in \N^*, \ \P (Z_{n+1}^y \in A_{n+1} \ \mid \  Z_n^y=z_n, \dots Z_1^y=z_1) = P (A \mid z_n, \dots z_1,y),
	\end{equation}
then
	\begin{equation}\label{eq:8.6.29}
	(Z_{n+k}^y,\ldots , Z_{n+k+l}^y) \sim \pi_{k,l}^{n,y}.
	\end{equation}
\begin{lem}
The sequence $\left( \pi_{k,l}^{(n,y)} (\cdot ) \right)_{n \in \N}$ is a Cauchy sequence in $\left(\mathcal{M}_1 \left( \X^{\otimes l+1}\right), d_{TV} \right)$.
\end{lem}

\begin{proof}

Let $n, p \in \N$ with $n+k \geqslant 1$, we have for any measurable function $h :X^{l+1} \longrightarrow [0,1]$
    \begin{align}
        \int h \ \d \pi_{k,l}^{n+p,y} - \int h \ \d \pi_{k,l}^{n,y} & = \E \left( h(Z_{n+p+k}^y,\ldots , Z_{n+p+k+l}^y) \right)- \E \left( h(Z_{n+k}^y,\ldots ,Z_{n+k+l}^y) \right) \notag \\
         & = \E \left( \E \left( h(Z_{n+p+k}^y,\ldots ,Z_{n+p+k+l}^y) \mid Z_p^y,\ldots ,Z_1^y\right) \right) - \E \left( h(Z_{n+k}^y,\ldots ,Z_{n+k+l}^y) \right).
    \end{align}
However, one can show that for any $z_1,\ldots ,z_p \in X$, we have
    \begin{equation} 
        \E \left( h(Z_{n+p+k}^y,\ldots Z_{n+p+k+l}^y) \mid Z_p^{y}=z_p,\ldots Z_1^y=z_1 \right) = \E \left( h(Z_{n+k}^{[z\cdot y]_p},\ldots , Z_{n+k+l}^{[z\cdot y]_p})\right)
    \end{equation}
where $[z \cdot y]_p = (z_p,\ldots, z_1,y)$.

Hence
    \begin{equation}\label{eq:8.6.32}
        \int h \ \d \pi_{k,l}^{n+p,y} - \int h \ \d \pi_{k,l}^{n,y} = \E \left( h(Z_{n+k}^{[Z^y\cdot y]_p},\ldots , Z_{n+k+l}^{[Z^y\cdot y]_p})\right) - \E \left( h(Z_{n+k}^y,\ldots ,Z_{n+k+l}^y) \right),
    \end{equation}
and by taking the expectancy in both sides of \ref{eq:8.6.32}, we obtain
    \begin{align} \label{eq:8.6.33}
        \int h \ \d \pi_{k,l}^{n+p,y} - \int h \ \d \pi_{k,l}^{n,y} \leqslant \E \left( d_{TV} \left( \pi_{k,l}^{n,[Z^y \cdot y]_p},\pi_{k,l}^{n,y} \right) \right).
    \end{align}

However, we have according to remarks \ref{eq:8.6.27} to \ref{eq:8.6.29} and the coupling inequality, we have for any $y,z \in X^{\N}$
	\begin{align} \label{eq:8.6.34}
	d_{TV}\left( \pi_{k,l}^{(n,y)} ; \pi_{k,l}^{(n,z)} \right) & \leqslant \P \left( \left( U_{n+k}^{y,z}, \dots , U_{n+k+l}^{y,z} \right) \neq \left( V_{n+k}^{y,z}, \dots ,V_{n+k+l}^{y,z} \right) \right) \notag \\
	 & = \P \left( \bigcup_{j=0}^{l} \left\lbrace U_{n+k+j}^{y,z} \neq V_{n+k+j}^{y,z} \right\rbrace \right) \notag \\
	 & \leqslant \sum_{j=0}^l \P \left( U_{n+k+j}^{y,z} \neq V_{n+k+j}^l \right) \notag \\
	 & \leqslant \sum_{j=0}^l b_{n+k+j}^*,
	\end{align}
where the last inequality is obtained by Lemma \ref{lem:8.2}.

Therefore, returning to equation \ref{eq:8.6.33} and taking the supremum on the functions $h$, we obtain
    \begin{equation}\label{eq:8.6.35}
        d_{TV} \left( \pi_{k,l}^{n+p,y},\pi_{k,l}^{n,y} \right) \leqslant \sum_{j=0}^l b_{n+k+j}^*.
    \end{equation}

According to Proposition 2 in \citet{bressaud1999decay}, since $b_0 <1$ and $\sum_{n \geqslant 0} b_n < +\infty$, we also have $\sum_{n \geqslant 0} b_n^* <+ \infty$, and for $n$ large enough, the last quantity in \ref{eq:8.6.35} is close to zero.

% Let us denote the signed measure $\mu_{k,l}^{n,p}=\pi_{k,l}^{(n,y)}-\pi_{k,l}^{(n+p,y)}$. Combining equations and , if $\varepsilon >0$ is fixed, there exists $n_0 \in \N$ such that for $n\geqslant n_0$, $p \in \N$ and any $A_k,\dots , A_{k+l} \in \X$:
% 	\begin{equation}
% 	\mu_{k,l}^{n,p}(A_k \times \dots \times A_{k+l} ) < \varepsilon.
% 	\end{equation}
% Yet, the measure $\mu_{k,l}^{n,p}$ is bounded, thus for any $B \in \X^{\otimes (l+1)}$, there exists $A^{(n)}=A_k^{(n)} \times \dots \times A_{k+l}^{(n)}$ in $\X^{l+1}$ such that:
% 	\begin{equation}
% 	\left| \mu_{k,l}^{n,p} \left( A^{(n)} \Delta B\right) \right| \leqslant \left|\mu_{k,l}^{n,p}\right| \left( A^{(n)} \Delta B \right) < \varepsilon.
% 	\end{equation}
% Thus, for $n\geqslant n_0$, $p\in \N$ and any $B \in \X^{\otimes (l+1)}$:
% 	\begin{equation}
% 	\left| \mu`_{k,l}^{n,p} (B)\right| \leqslant \left| \mu_{k,l}^{n,p} (B) - \mu_{k,l}^{n,p} (A^{(n)}) \right| + \left| \mu_{k,l}^{n,p} (A^{(n)}) \right| < 2\varepsilon
% 	\end{equation}
% which means that for $n\geqslant n_0$ and $p \in \N$:
% 	\begin{equation}
% 	d_{TV} \left( \pi_{k,l}^{(n,y)} ; \pi_{k,l}^{(n+p,y)} \right) \leqslant 2 \varepsilon.
% 	\end{equation}
Thus, $\left( \pi_{k,l}^{(n,y)} \right)_{n \geqslant 0}$ is a Cauchy sequence in $\left( \mathcal{M}_1(\X^{\otimes (l+1)}), d_{TV} \right)$.
\end{proof}

%%%% FIN MODIFICATION

Since $\left( \mathcal{M}_1 \left( \X^{\otimes (l+1)} \right), d_{TV} \right)$ is complete, there exists for $k \in \Z,\ l\in \N$ and $y \in X^{\N}$ fixed, a probability measure $\pi_{k,l}^y$ such that:
	\begin{equation}
	\pi_{k,l}^{(n,y)} \underset{d_{TV}}{\longrightarrow} \pi_{k,l}^y.
	\end{equation}
\begin{lem}\label{convuni}
There exists $\pi_{k,l} \in \mathcal{M}_1( \X^{\otimes (l+1)} )$ such that:
	\begin{equation}
	\underset{y \in X^{\N}}{\sup} \ d_{TV} \left( \pi_{k,l}^{(n,y)} ; \pi_{k,l}\right) \underset{n \to +\infty}{\longrightarrow} 0.
	\end{equation}
\end{lem}

\begin{proof}
Take any $y,z \in X^{\N}$, we have for all $n \in \N$:
	\begin{equation}
	d_{TV} \left( \pi_{k,l}^y ; \pi_{k,l}^z \right) \leqslant \underbrace{d_{TV} \left( \pi_{k,l}^y ; \pi_{k,l}^{(n,y)} \right)}_{\underset{n \to +\infty}{\longrightarrow} 0} + d_{TV} \left( \pi_{k,l}^{(n,y)} ; \pi_{k,l}^{(n,z)} \right) + \underbrace{d_{TV} \left( \pi_{k,l}^{(n,z)} ; \pi_{k,l}^{z} \right)}_{\underset{n \to +\infty}{\longrightarrow} 0}.
	\end{equation}
Moreover, we already proved with inequality \ref{eq:8.6.34} that $d_{TV} \left( \pi_{k,l}^{(n,y)} ; \pi_{k,l}^{(n,z)} \right) \underset{n \to +\infty}{\longrightarrow} 0$.

Thus, for any $y, z \in X^{\N}$, we have $\pi_{k,l}^y = \pi_{k,l}^z$. Let us denote $\pi_{k,l} = \pi_{k,l}^{z}$ for some $z$ fixed and $\tilde{x}=(x,x,\dots ) \in X^{\N}$ for an element $x \in X$ also fixed.

Let $\varepsilon >0$, there exists $n_0 \in \N$ such that for all $n\geqslant n_0$:
	\begin{equation}
	\sum_{j=0}^l b_{n+k+j}^* < \varepsilon \quad \text{and} \quad d_{TV} \left( \pi_{k,l}^{(n,\tilde{x})} ; \pi_{k,l} \right) < \varepsilon.
	\end{equation}
Thus, according to inequality \ref{eq:8.6.34}, we have for any $y \in X^{\N}$ and $n \geqslant n_0$:
	\begin{align}
	d_{TV} \left( \pi_{k,l}^{(n,y)} ; \pi_{k,l} \right) & \leqslant d_{TV} \left( \pi_{k,l}^{(n,y)} ; \pi_{k,l}^{(n,\tilde{x})} \right) + d_{TV} \left( \pi_{k,l}^{(n,\tilde{x})} ; \pi_{k,l} \right) \notag \\
	& \leqslant \sum_{j=0}^l b_{n+k+j}^* + \varepsilon \notag \\
	& < 2\varepsilon.
	\end{align}
So, for $n \geqslant n_0$, we have:
	\begin{equation}
	\underset{ y \in X^{\N}}{\sup} \ d_{TV} \left( \pi_{k,l}^{(n,y)} ; \pi_{k,l} \right) \leqslant 2\varepsilon.
	\end{equation}
Hence the result.
\end{proof}
It is therefore possible to define a probability measure $\pi$ on $\X^{\Z}$ by defining it for any cylinder $A = \prod_{n \in \Z} A_n$ of $X ^{\Z}$:
	\begin{equation}
	\pi (A) = \pi_{k,l-k} (A_k \times \dots A_{l})
	\end{equation}
where $k = \inf \ \left\lbrace n \in \Z \ \bigm\mid \ A_n \neq X \right\rbrace$ and $l = \sup \ \left\lbrace n \in \Z \ \bigm\mid \ A_n \neq X \right\rbrace$.

\medskip

As a consequence, there exists a stochastic process $(Y_t)_{t\in \Z}$ valued in $X$ with distribution $\P_Y = \pi$.

The first natural property to verify is if $(Y_t)_{t\in \Z}$ satisfies equation \ref{eq:2.1.10}, as it is proven in the following lemma.

\begin{lem}
For all $t \in \Z$ and $A \in \X$:
	\begin{equation}
	\P\left( Y_t \in A \ \mid \ Y_{t-1}^-=y_{t-1}^-\right) = P (A \mid y_{t-1}^-).
	\end{equation}
\end{lem}

\begin{proof}
Consider $t \in \Z$ and $k \in \N$ fixed.

Let us denote by $\pi_t^-$ the distribution of $Y_t^-$, and for any $A_{0}, \dots , A_{0k} \in \X$, let us define the function
	\begin{equation}
	h : \begin{array}[t]{rl}
	X^{k+1} & \longrightarrow [0;1] \\
	(y_0,\dots ,y_{k}) & \longmapsto \mathds{1}_{A_0}(y_t) \mathds{1}_{A_{1}} (y_{1}) \dots \mathds{1}_{A_{k}}(y_{k}).
\end{array}
	\end{equation}
In a first instance, we aim to prove that
	\begin{equation} \label{eq:8.6.45}
	\int h(y_t, \dots ,y_{t-k}) \ \pi_t^- (\d y_t^-) = \int \int h(y_t, \dots , y_{t-k}) P(\d y_t \mid y_{t-1}^-) \ \pi_{t-1}^- (\d y_{t-1}^-)
	\end{equation}
i.e.
	\begin{equation} \label{eq:8.6.46}
	\P (Y_t \in A_0 , Y_{t-1}^- \in B_k) = \int_{B_k} P(A_0 \mid y_{t-1}^-) \pi_{t-1}^-(\d y_{t-1}^-),
	\end{equation}
where $B_k$ is the cylinder of $X^{\N}$: $B_k = A_{1} \times \dots \times A_{k} \times X \times \dots $.

Let $\varepsilon >0$, for $m  \in \N$ large enough, $m\geqslant k$, we have $b_m < \varepsilon$ and:
	\begin{equation}
	\forall n \geqslant m, \ d_{TV} \left( \pi_{t-k,k}; \pi_{t-k,k}^{(n,\tilde{x})} \right) + d_{TV}\left( \pi_{t-m,m-1} ; \pi_{t-m,m-1}^{(n,\tilde{x})}\right) < \varepsilon
	\end{equation}
where $\tilde{x}=(x,x,\dots )$ for a fixed element $x \in X$.

In the following, we take any $n \geqslant m$, so that the previous inequality is satisfied.

Let us define
	\begin{equation}
	C_m = \int \int h(y_{n+t},\dots y_{n+t-k}) P(\d y_{n+t} \mid y_{n+t-1}, \dots , y_{n+t-m}, \tilde{x}) \pi_{t-m,m-1}^{(n,\tilde{x})} (\d y_{n+t-m}, \dots ,\d y_{n+t-1}).
	\end{equation}
We have
	\begin{align}
 	& \ \left| C_m - \int h(y_{n+t}, \dots , y_{n+t-k}) \pi_{t-k,k}^{(n,\tilde{x})}(\d y_{n+t-k}, \dots ,\d y_{n+t}) \right| \notag \\
 	= &\ \left| \int \int h(y_{n+t}, \dots ,y_{n+t-k}) P(\d y_{n+t} \mid y_{n+t-1}, \dots y_{n+t-m},\tilde{x})\pi_{t-m,m-1}^{(n,\tilde{x})}(\d y_{n+t-m},\dots ,\d y_{n+t-1}) \right. \notag \\
 	& \hspace*{8cm} \left. - \int h(y_{n+t},\dots ,y_{n+t-k}) \pi_{t-k,k}^{(n,\tilde{x})}(\d y_{n+t-k}, \dots ,\d y_{n+t})\right| \notag \\
= & \ \Biggr| \underbrace{\int \dots \int}_{n+t-1} \left[ \int h(y_{n+t},\dots ,y_{n+t-k}) P(\d y_{n+t} \mid y_{n+t-1},\dots ,y_{n+t-m},\tilde{x})\right. \notag \\
 & \ - \left.\int h(y_{n+t},\dots ,y_{n+t-k}) P(\d y_{n+t} \mid y_{n+t-1}, \dots ,y_{1},\tilde{x}) \right] P(\d y_{n+t-1} \mid y_{n+t-2}, \dots ,y_{1},\tilde{x})\dots P(\d y_{1} \mid \tilde{x}) \Biggr| \notag \\
\leqslant & \underbrace{\int \dots \int}_{n+t-1} b_m P(\d y_{n+t-1} \mid y_{n+t-2}, \dots , y_{1},\tilde{x}) \dots P(\d y_{1} \mid \tilde{x}) \notag \\
= & \ b_m \notag \\
< & \ \varepsilon.
	\end{align}
Furthermore
	\begin{align}
	 & \left| C_m - \int \int h(y_{t},\dots ,y_{t-k}) P(\d y_{t} \mid y_{t-1}, \dots ,y_{t-m},\tilde{x}) \pi_{t-m,m-1}(\d y_{t-m},\dots ,\d y_{t-1})\right| \notag \\
	  \leqslant & \ d_{TV} \left( \pi_{t-m,m-1}^{(n,\tilde{x})};\pi_{t-m,m-1}\right) \notag \\
	   < & \varepsilon.
	\end{align}
Finally, we have for any $y_{t-1}^- \in X^{\N}$
	\begin{equation}
	\left| \int h(y_{t},\dots , y_{t-k}) P(\d y_t \mid y_{t-1}, \dots ,y_{t-m},\tilde{x}) - \int h(y_{t}, \dots , y_{t-k}) P(\d y_{t} \mid y_{t-1}^-) \right|  \leqslant b_m < \varepsilon,
	\end{equation}
which implies
	\begin{align}
& \ \Biggr| \int \int h(y_{t},\dots , y_{t-k}) P(\d y_{t} \mid y_{t-1},\dots ,y_{t-m},\tilde{x})\pi_{t-m,m-1}(\d y_{t-m},\dots ,\d y_{t-1}) \notag \\
 & \hspace*{6cm} -\int \int h(y_{t},\dots ,y_{t-k}) P(y_{t} \mid y_{t-1}^-) \pi_{t-1}^-(\d y_{t-1}^-) \Biggr| \notag \\
= & \ \Biggr| \int \int h(y_t,\dots , y_{t-k}) P( \d y_t \mid y_{t-1},\dots ,y_{t-m},\tilde{x})\pi_{t-1}^-(\d y_{t-1}^-)-\int \int h(y_t,\dots ,y_{t-k}) P(\d y_t \mid y_{t-1}^-) \pi_{t-1}^-(\d y_{t-1}^-) \Biggr| \notag \\
\leqslant & \ \int \left| \int h(y_t,\dots ,y_{t-k}) P( \d y_t \mid y_{t-1},\dots , y_{t-m},\tilde{x}) - \int h(y_t,\dots ,y_{t-k}) P(\d y_t \mid y_{t-1}^-) \right| \pi_{t-1}^-(\d y_{t-1}^-) \notag \\
\leqslant & \ \varepsilon.
	\end{align}
We thus obtain
	\begin{align}
	 & \ \left| \int h(y_t,\dots ,y_{t-k})\pi_{t-k,k} (\d y_{t-k},\dots \d y_t) - \int \int h(y_t,\dots y_{t-k})P(\d y_t \mid y_{t-1}^-) \pi_{t-1}^- (\d y_{t-1}^-)\right| \notag \\
\leqslant & \ \left| \int h(y_t,\dots y_{t-k}) \pi_{t-k,k} (\d y_{t-k},\dots ,\d y_t) - \int h(y_t,\dots y_{t-k}) \pi_{t-k,k}^{(n,\tilde{x})}(\d y_{t-k},\dots \d y_t) \right| \notag \\
 & \quad + \left| \int h(y_{n+t},\dots ,y_{n+t-k})\pi_{t-k,k}^{(n,\tilde{x})}(\d y_{n+t-k}, \dots ,\d y_{n+t}) -C_m \right| \notag \\
 & \quad + \left| C_m - \int \int h(y_t,\dots y_{t-k}) P(\d y_t \mid y_{t-1},\dots y_{t-m},\tilde{x})\pi_{t-m,m-1}(\d y_{t-m},\dots \d y_{t-1}) \right| \notag \\
 & \ \quad + \Biggr| \int \int h(y_t,\dots y_{t-k}) P (\d y_t \mid y_{t-1},\dots y_{t-m},\tilde{x})\pi_{t-m,m-1}(\d y_{t-m},\dots ,\d y_{t-1}) \notag \\
  & \hspace*{6cm} - \int \int h(y_t,\dots y_{t-k}) P(\d y_t \mid y_{t-1}^-) \pi_{t-1}^-(\d y_{t-1}^-) \Biggr| \notag \\
\leqslant & \  4\varepsilon.
	\end{align}
Since $\varepsilon >0$ is arbitrary, this gives us equations \ref{eq:8.6.45} and \ref{eq:8.6.46}.

By an argument of monotone class, equation \ref{eq:8.6.46} then leads to:
	\begin{equation}
	\forall A \in \X, \ \forall B \in \X^{\otimes \N}, \ \P(Y_t \in A, Y_{t-1}^- \in B ) = \int_B P(A  \mid y_{t-1}^-) \pi_{t-1}^-(\d y_{t-1}^-),
	\end{equation}
which means that:
	\begin{equation}
	\forall A \in \X, \ \forall y_{t-1}^- \in X^{\N}, \ \P(Y_t \in A  \ \mid \ Y_{t-1}^-=y_{t-1}^-)= P(A \mid y_{t-1}^-).
	\end{equation}
\end{proof}

\begin{lem}

The process $(Y_t)_{t\in \Z}$ is strictly stationary.

\end{lem}

\begin{proof}
Let $t \in \Z$ and $k \in \N^*$, one can show that for any $n \in \N$ such that $n+t \in \N$
	\begin{equation}
	\pi_{t+1,k-1}^{(n, \tilde{x})} = \pi_{1,k-1}^{(n+t,\tilde{x})},
	\end{equation}
where $\tilde{x}=(x,x,\dots )$ for a fixed $x \in X$.

Therefore, we have for any $h \in \mathcal{C}_b(X^k, \R)$
	\begin{align}
	\E \left( h(Y_{t+1},\dots , Y_{t+k}) \right) & \int h(y_1,\dots , y_k) \pi_{t+1,k-1}(\d y_{1},\dots , \d y_{k}) \notag \\
	 & = \lim_{n \to +\infty} \int h(y_1, \dots y_k) \pi_{t+1,k-1}^{(n,\tilde{x})} (\d y_1 , \dots , \d y_k) \notag \\
	 & = \lim_{n \to +\infty} \int h(y_1,\dots ,y_k) \pi_{1,k-1}^{(n+t,\tilde{x})}(\d y_1 ,\dots , \d y_k) \notag \\
	 & = \int h(y_1,\dots , y_k) \pi_{1,k-1}(\d y_1, \dots , \d y_k) \notag \\
	 & = \E \left( h (Y_1,\dots , Y_k) \right).
	\end{align}
Indeed, when considering a Borel $\sigma$-field, the convergence in total variation implies the convergence in distribution.

We thus obtain the desired result.
\end{proof}

We now show the uniqueness of a process satisfying the desired conditions mentioned in Theorem \ref{theo:2.1}.

\begin{lem}

The process $(Y_t)_{t\in \Z}$ is unique in the sense that, if $(Z_t)_{t \in \Z}$ is another strictly stationary process such that
	\begin{equation}
	\forall t \in \Z, \forall A \in \X, \  \ \P \left( Z_t \in A \ \mid \ Z_{t-1}^- = z_{t-1}^- \right) = P(A \mid z_{t-1}^-),
	\end{equation}
then $\P_Z = \pi$.

\end{lem}

\begin{proof}
Since $(Z_t)_{t\in \Z}$ is strictly stationary, it is sufficient to show that for all $k \in \N^*$
	\begin{equation}
	(Z_1,\dots , Z_k) \sim \pi_{1,k-1}.
	\end{equation}
Consider $h \in \mathcal{C}_b (X^k , \R)$, we have for all $n \in \N$
	\begin{equation}
	\E \left( h(Z_1,\dots , Z_k) \right) = \E \left( \E \left( h(Z_1,\dots , Z_k) \ \mid Z_{-n}^- \right) \right).
	\end{equation}
However
	\begin{equation}
	\E \left( h(Z_1,\dots , Z_k) \mid Z_{-n}^-\right) = \underbrace{\int_X \dots \int_X}_{n+k} h(z_1,\dots , z_k) P(\d z_k \mid z_{k-1},\dots , z_{-n+1},Z_{-n}^-) \dots P(\d z_{-n+1} \mid Z_{-n}^-),
	\end{equation}
thus
	\begin{equation}
	\E \left( \E \left( h(Z_1,\dots , Z_k) \ \mid \ Z_{-n}^- \right) \right) = \E \left( \int h(z_1,\dots , z_k) \pi_{1,k-1}^{(n,Z_{-n}^-)}(\d z_1,\dots , \d z_k) \right).
	\end{equation}
By Lemma \ref{convuni} we have $ \displaystyle{ \underset{y \in \X^{\N}}{\sup} \ d_{TV} \left( \pi_{1,k-1}^{(n,y)} ; \pi_{1,k-1} \right) \underset{n \to +\infty}{\longrightarrow} 0}$, thus
	\begin{equation}
	\lim_{n \to +\infty} \int h(z_1,\dots , z_k) \pi_{1,k-1}^{(n,Z_{-n}^-)}(\d z_1,\dots, \d z_k) = \int h(z_1,\dots, z_k) \pi_{1,k-1}(\d z_1,\dots , \d z_k),
	\end{equation}
since the convergence in total variation implies the convergence in distribution.

Furthermore, since $h$ is a bounded function, we have
	\begin{equation}
	\int h(z_1,\dots , z_k) \pi_{1,k-1}^{(n,Z_{-n}^-)}(\d z_1,\dots, \d z_k) \leqslant \int \| h\|_{\infty} \pi_{1,k-1}^{(n,Z_{-n}^-)} (\d z_1, \dots \d z_k) = \| h \|_{\infty}.
	\end{equation}
Therefore, we have by dominated convergence
	\begin{equation}
	\lim_{n \to +\infty} \E \left( \E \left( h(Z_1,\dots , Z_k) \ \mid \ Z_{-n}^- \right) \right) = \E \left( h(Y_1,\dots , Y_k) \right)=\int h d\pi_{1,k-1},
	\end{equation}
which is the desired result.
\end{proof}

We end the proof of Theorem \ref{theo:2.1} with the following lemma.

\begin{lem}
The process $(Y_t)_{t\in \Z}$ is ergodic.
\end{lem}

\begin{proof}
Consider two cylinders of $\X^{\Z}$
	\begin{equation}
	A=\dots X \times A_k \times \dots \times A_{k+l} \times X \dots
	\end{equation}
and
	\begin{equation}
	B=\dots X \times B_s \times \dots \times B_{s+t} \times X \dots
	\end{equation}
where $k,s \in \Z$ and $l, t \in \N$.

Let $n, m \in \N$ such that $n >-k$, we have for $\tilde{x}=(x,x,\dots)$ a fixed element in $X^{\N}$
	\begin{align} \label{eq:8.6.68}
	 & \pi_{k,l+m+t+1}^{(n,\tilde{x})}(A_k \times \dots \times A_{k+l} \times \underbrace{X \times \dots \times X}_{m} \times B_s\times \dots \times B_{s+t}) \notag \\	 
	= & \underbrace{\int_X \dots \int_X}_{n+k-1} \int_{A_k} \dots \int_{A_{k+l}} \underbrace{\int_X \dots \int_X}_{m} \int_{B_s} \dots \int_{B_{s+t}} P(\d y_{n+k+l+m+t+1} \mid y_{n+k+l+m+t},\dots , y_{1},\tilde{x}) \notag \\
	 & \hspace*{2cm} \dots P(\d y_{n+k+l+1} \mid y_{n+k+l},\dots ,y_{1},\tilde{x}) P(\d y_{n+k+l} \mid y_{n+k+l-1},\dots , y_{1},\tilde{x})\dots P(\d y_{1} \mid \tilde{x}) \notag \\
	= & \int_X \dots \int_X \int_{A_k} \dots \int_{A_{k+l}} \pi_{1,t}^{(m,\tilde{y})}(B_s \times \dots B_{s+t}) P(\d y_{n+k+l} \mid y_{n+k+l-1},\dots ,y_{1},\tilde{x}) \dots P(\d y_{1} \mid \tilde{x}) \notag \\
        & \hspace*{8cm} \text{where} \ \tilde{y} = (y_{n+k+l},\ldots ,y_1,\tilde{x}) \notag \\
	\underset{m \to +\infty}{\longrightarrow} & \int_X \dots \int_X \int_{A_k}\dots \int_{A_{k+l}}P(\d y_{n+k+l} \mid y_{n+k+l-1},\dots , y_{1}, \tilde{x})\dots P(\d y_{1} \mid \tilde{x}) \times \pi_{1,t}(B_s\times \dots \times B_{s+t}) \notag \\
	& \hspace*{8cm} \text{by dominated convergence} \notag \\
	= & \pi_{k,l}^{(n,\tilde{x})}(A_k\times \dots \times A_{k+l})\times \pi_{s,t}(B_s\times \dots \times B_{s+t}) \hspace*{1cm} \text{since} \ \pi_{1,t} = \pi_{s,t} \notag \\
	\underset{n\to +\infty}{\longrightarrow} & \pi_{k,l}(A_k\times \dots \times A_{k+l})\times \pi_{1,t}(B_s\times \dots \times B_{s+t}) \notag \\
	= & \pi(A) \times \pi(B).
	\end{align}
However
	\begin{align} \label{eq:8.6.69}
	\pi_{k,l+m+t+1}^{(n,\tilde{x})}(A_k \times \dots \times A_{k+l} \times \underbrace{X \times \dots \times X}_{m} \times B_s\times \dots \times B_{s+t}) \notag \\
	\underset{n \to + \infty}{\longrightarrow} \underbrace{\pi_{k,l+m+t+1}(A_k\times \dots \times A_{k+l}\times X \times \dots \times X \times B_{s} \times \dots \times B_{s+t})}_{=\pi\left( A \cap \sigma^{-m} ( B)\right)}.
	\end{align}
 Defining the shift operator $\sigma$ on $X^{\Z}$ by
$$\sigma \cdot y = \left( y_{t+1}\right)_{t\in \Z},$$
and combining equations \ref{eq:8.6.68} and \ref{eq:8.6.69} we obtain
	\begin{equation}
	\lim_{m \to +\infty} \pi \left(A \cap \sigma^{-m} (B)\right) = \pi(A) \times \pi(B)
	\end{equation}
for any cylinders $A, B \in \X^{\Z}$.

Using approximations of measurable sets by cylinders, we obtain the same result for any $A, B$ in $\X^{\otimes \Z}$.

Consider now $A \in \X^{\otimes \Z}$ invariant for $\pi$, that is $\sigma (A) = A$. We have for any $m \in \N$
	\begin{equation}
	\pi(A) = \pi(A \cap \sigma^{-m}(A)) \underset{m \to +\infty}{\longrightarrow} \pi(A)^2.
	\end{equation}
Thus, $\pi(A) \in \left\lbrace 0 , 1 \right\rbrace$, and $(Y_t)_{t\in \Z}$ is ergodic.
\end{proof}

\begin{lem}\label{DirichletLip}
For some $\alpha\in (0,\infty)^d$, let $\nu^{(d)}_{\alpha}$ be the Dirichlet distribution of parameter $\alpha$. For $1\leq i\leq d$, let 
$0<m_i<M_i$. There then exists $K>0$ such that 
			\begin{equation}
				\forall \alpha, \alpha' \in \prod_{i=1}^d [m_i,M_i], \ d_{TV}\left( \nu_{\alpha} ; \nu_{\alpha'}\right) \leqslant K \vert \alpha - \alpha'\vert.
			\end{equation}
\end{lem}
	
\begin{proof}
Without loss of generality, we assume that $\vert\cdot\vert$ is the infinite norm $\Vert\cdot\Vert_{\infty}$.
	Let us first recall that if $\gamma_1, \ \gamma_2$ are two probability measures with $\gamma_1 = h_1 \cdot \nu$ and $\gamma_2=h_2\cdot \nu$ where $\nu$ is a positive measure, then
		\begin{equation}
			d_{TV}(\gamma_1 , \gamma_2 ) = \dfrac{1}{2} \int | h_1-h_2 | \ \d \nu.
		\end{equation}
	Consequently, we have for any $\alpha, \alpha'$ in $\prod_{i=1}^d [m_i ; M_i]$:
	\begin{multline}
		d_{TV}\left( \nu_{\alpha}^{(d)} ; \nu_{\alpha'}^{(d)} \right) = \frac{1}{2} \mathlarger{\mathlarger{\mathlarger{\int}}}_{B_{d-1}} \left| \dfrac{\Gamma\left( \sum_{i=1}^d \alpha_i \right)}{\Gamma(\alpha_1)\dots \Gamma (\alpha_d)} \prod_{i=1}^{d-1} x_i^{\alpha_i-1} \left( 1- \sum_{i=1}^{d-1}x_i\right)^{\alpha_d-1}\right. \\  - \left. \dfrac{\Gamma\left( \sum_{i=1}^d \alpha'_i \right)}{\Gamma(\alpha'_1)\dots \Gamma (\alpha'_d)} \prod_{i=1}^{d-1} x_i^{\alpha'_i-1} \left( 1- \sum_{i=1}^{d-1}x_i\right)^{\alpha'_d-1} \right| \ \d x_1 \dots \d x_{d-1}.
	\end{multline}
	
	For $x \in B_{d-1}$ fixed, let us denote the application:
		\begin{equation}
			\psi_x : \begin{array}[t]{l}
				\prod_{i=1}^d [m_i ; M_i] \longrightarrow \R_+ \\
				\alpha \longmapsto \dfrac{\Gamma\left( \sum_{i=1}^d \alpha_i \right)}{\Gamma(\alpha_1)\dots \Gamma (\alpha_d)} \prod_{i=1}^{d-1} x_i^{\alpha_i-1} \left( 1- \sum_{i=1}^{d-1}x_i\right)^{\alpha_d-1}
			\end{array}
		\end{equation}
	which is $\mathcal{C}^1$, and we have:
		\begin{equation}
			d_{TV}\left( \nu_{\alpha}^{(d)} ; \nu_{\alpha'}^{(d)} \right) = \dfrac{1}{2} \int_{B_{d-1}} \left| \psi_x(\alpha) - \psi_x(\alpha')\right| \ \d x_1 \dots \d x_{d-1}.
		\end{equation}
	Take any $j \in \left\lbrace 1 , \dots d-1 \right\rbrace$, denoting $S_{\alpha} = \sum_{i=1}^d \alpha_i$ and $x_d=1-\sum_{i=1}^{d-1} x_i$, we have:
		\begin{equation}
			\dfrac{\partial \psi_x}{\partial \alpha_j}(\alpha) = \prod_{i=1}^d x_i^{\alpha_i-1} \times \left( \dfrac{\Gamma'(S_{\alpha})\Gamma(\alpha_1) \dots \Gamma (\alpha_d)-\Gamma (S_{\alpha})\Gamma'(\alpha_j)\prod_{i \neq j}\Gamma(\alpha_i)}{\left( \Gamma (\alpha_1) \dots \Gamma (\alpha_d)\right)^2}+ \dfrac{\Gamma(S_{\alpha})\log (x_j)}{\Gamma(\alpha_1)\dots \Gamma (\alpha_d)}\right).
		\end{equation}
	Because the Gamma function $\Gamma$ is $\mathcal{C}^{\infty}$ on $]0 ; +\infty [$ and $\prod_{i=1}^d [m_i ; M_i]$ is compact, there exists a couple of positive numbers $(M_1,M_2)$ such that:
		\begin{align}
			\left| \dfrac{\partial \psi_x}{\partial \alpha_j}(\alpha) \right| & \leqslant \prod_{i=1}^d x_i^{\alpha_i-1} \times \left| M_1+M_2 \log (x_j)\right| \\
			& \leqslant \underbrace{M_1\prod_{i=1}^dx_i^{m_i-1}+M_2\prod_{i=1}^d x_i^{m_i-1}\cdot \left| \log (x_j)\right|}_{C_j(x)}.
		\end{align}
	Note that a similar inequality is also satisfied for $j=d$.
	
	By the mean value inequality, we obtain
		\begin{equation}
			\left| \psi_x(\alpha)-\psi_x(\alpha')\right| \leqslant \underset{\alpha}{\sup} \ \vertiii{D_{\psi_x}(\alpha)} \cdot \| \alpha - \alpha'\|_{\infty}
		\end{equation}
	where $D_{\psi_x}$ denotes the differential of $\psi_x$ and:
		\begin{equation}
			\vertiii{D_{\psi_x}(\alpha)} = \underset{h \in \R^d}{\sup} \ \dfrac{\left| \sum_{j=1}^d \dfrac{\partial \psi_x}{\partial \alpha_j}(\alpha) \cdot h_j\right|}{\| h\|_{\infty}} \leqslant \underset{h\in \R^d}{\sup} \ \dfrac{\sum_{j=1}^d C_j(x)\cdot |h_j|}{\| h \|_{\infty}} \leqslant \sum_{j=1}^d C_j (x).
		\end{equation}	
	Therefore,
		\begin{equation}
			d_{TV}\left(\nu_{\alpha}^{(d)} ; \nu_{\alpha'}^{(d)}\right) \leqslant \dfrac{\|\alpha - \alpha'\|_{\infty}}{2} \mathlarger{\int}_{B_{d-1}} \sum_{j=1}^d C_j(x) \ \d x_1 \dots \d x_{d-1}.
		\end{equation}
	Our goal now is to show that for any $j \in \left\lbrace 1 , \dots , d \right\rbrace, \ \int_{B_{d-1}} C_j(x) \d x_1 \dots \d x_{d-1} < + \infty$.
	
	We then have the desired result by taking $K = 1+\dfrac{1}{2}\sum_{j=1}^d \int_{B_{d-1}} C_j(x) \d x_1 \dots \d x_{d-1}$ for instance.
	
	\medskip
	
	For any $j \in \left\lbrace 1, \dots d \right\rbrace$, we have:
	
	\begin{align}
		\int_{B_{d-1}} C_j(x) \d x_1 \dots \d x_{d-1} & = M_1\int_{B_{d-1}} \prod_{i=1}^d x_i^{m_i-1} \d x_1 \dots \d x_{d-1} + M_2 \int_{B_{d-1}} \prod_{i=1}^d x_i^{m_i-1} \left| \log (x_j) \right| \d x_1 \dots \d x_{d-1} \\
		& = M_1 \dfrac{\Gamma(m_1) \dots \Gamma (m_d)}{\Gamma \left( \sum_{i=1}^d m_i \right)}+ M_2 \underbrace{\int_{B_{d-1}} \prod_{i=1}^d x_i^{m_i-1} \left| \log (x_j) \right| \d x_1 \dots \d x_{d-1}}_{I_j^{(d)}(m_1,\dots , m_d)}.
	\end{align}
	
	We thus have to show by induction the following property:
	
	\begin{equation}
	\forall d \geqslant 2, \ H_d : \forall j \in \left\lbrace 1, \dots , d \right\rbrace, \forall m_1,\dots , m_d >0, \ I_j^{(d)}(m_1,\dots , m_d) <+ \infty.
	\end{equation}
	
	\begin{enumerate}[label=$\triangleright$]
		\item $H_2$ is true, we have indeed
			\begin{equation}
				I_1^{(2)} = \int_0^1 x_1^{m_1-1}(1-x_1)^{m_2-1} | \log (x_1)| \ \d x_1.
			\end{equation}
		Yet, on the neighborhood of 0:
			\begin{equation}
				x_1^{m_1-1}(1-x_1)^{m_2-1}|\log(x_1)| \sim x_1^{m_1-1} |\log(x_1)|
			\end{equation}
		and $\int_0 ^1 x_1^{m_1-1} | \log(x_1) |\ \d x_1$ is a convergent Bertrand integral.
		On the neighborhood of 1:
			\begin{equation}
				x_1^{m_1-1}(1-x_1)^{m_2-1}|\log (x_1)| \sim (1-x_1)^{m_2-1} | \log (x_1) |
			\end{equation}
		and $\int_{\frac{1}{2}}^1 (1-x_1)^{m_2-1} | \log (x_1) | \ \d x_1 \leqslant \log (2) \int_{\frac{1}{2}}^1 (1-x_1)^{m_2-1}\ \d x_1 <+\infty$.
		
		Thus, $I_1^{(2)}(m_1,m_2)<+\infty$.
		
		We show that $I_2^{(2)} <+\infty$ by noticing that:
			\begin{equation}
				I_2^{(2)} = \int_0^1 x_1^{m_1-1}(1-x_1)^{m_2-1}| \log (1-x_1)| \ \d x_1 = \int_0^1 (1-u)^{m_1-1}u^{m_2-1}| \log (u)| \ \d u = I_1^{(2)}(m_2,m_1)
			\end{equation}
		and this last quantity is finite.
		\item Assume that $H_d$ is true for some $d \geqslant 2$ fixed, and let us show that $H_{d+1}$ is also true.
		
		Let $m_1, \dots , m_{d+1} >0$ and $ j \in \left\lbrace 1, \dots , d \right\rbrace$, we have
		\begin{equation}
			I_j^{(d+1)}(m_1,\dots ,m_{d+1}) = \int_{B_{d}} \prod_{i=1}^d x_i^{m_i-1}\left( 1- \sum_{i=1}^d x_i \right)^{m_{d+1}-1} | \log(x_j)| \ \d x_1 \dots \d x_d.
		\end{equation}
	Without any loss of generality, we can set $j=1$ event if it means permuting the indices, and we have:
	\begin{align}
		& I_1^{(d+1)}(m_1,\dots ,m_{d+1}) \notag \\
		=& \int_{B_d} x_1^{m_1-1}\left( 1- \sum_{i=2}^dx_i -x_1 \right)^{m_{d+1}-1} | \log(x_1)| \prod_{i=2}^d x_i^{m_i-1} \ \d x_1 \dots \d x_d \notag \\
		=& \int_{B_d} x_1^{m_1-1} \left( 1- \sum_{i=2}^d x_i\right)^{m_{d+1}-1}\left( 1 - \dfrac{x_1}{1-\sum_{i=2}^d x_i}\right)^{m_{d+1}-1} \times |\log (x_1)| \prod_{i=2}^d x_i^{m_i-1} \ \d x_1\dots \d x_d \notag \\
		=& \int_{B_{d-1}} \prod_{i=2}^d x_i^{m_i-1}\left( 1-\sum_{i=2}^dx_i\right)^{m_{d+1}-1} \notag \\
		& \hspace*{1cm} \times \int_0^{1-\sum_{i=2}^d x_i} x_1^{m_1-1}\left( 1- \dfrac{x_1}{1-\sum_{i=2}^dx_i}\right)^{m_{d+1}-1} |\log (x_1)| \d x_1 \  \d x_2 \dots \d x_d \notag \\
		=& \int_{B_{d-1}} \prod_{i=2}^d x_i^{m_i-1}\left( 1-\sum_{i=2}^d x_i\right)^{m_1+m_{d+1}-1} \notag \\
		& \hspace*{1cm} \times \int_0^1 u^{m_1-1}(1-u)^{m_{d+1}-1}\left| \log \left( \left(1-\sum_{i=2}^d x_i \right)u\right)\right| \d u \ \d x_{2} \dots \d x_d \notag \\
		\leqslant &  \int_{B_{d-1}} \prod_{i=2}^d x_i^{m_i-1}\left( 1- \sum_{i=2}^d x_i\right)^{m_1+m_{d+1}-1}\left| \log \left( 1-\sum_{i=2}^d x_i\right) \right| \int_0^1 u^{m_1-1}(1-u)^{m_{d+1}-1} \d u \ \d x_2 \dots \d x_d \notag \\
		& + \int_{B_{d-1}} \prod_{i=2}^d x_i^{m_i-1}\left( 1- \sum_{i=2}^d x_i\right)^{m_1+m_{d+1}-1}\int_0^1 u^{m_1-1}(1-u)^{m_{d+1}-1}|\log(u)|\d u \ \d x_2 \dots \d x_d \notag \\
		= & \beta(m_1,m_{d+1}) \int_{B_{d-1}} \prod_{i=2}^d x_i^{m_i-1}\left( 1- \sum_{i=2}^d \right)^{m_1+m_{d+1}-1}\left| \log \left( 1-\sum_{i=2}^d x_i\right) \right| \d x_2 \dots \d x_d \notag \\
		& + I_1^{(2)}(m_1,m_{d+1}) \int_{B_{d-1}} \prod_{i=2}^d x_i^{m_i-1}\left( 1-\sum_{i=2}^d x_i \right)^{m_1+m_{d+1}-1}\d x_2 \dots \d x_d \notag \\
		= & \beta(m_1,m_{d+1})\times I_d^{(d)}(m_2,\dots , m_d,m_1+m_{d+1}) + I_1^{(2)}(m_1,m_{d+1})\times \dfrac{\Gamma(m_1+m_{d+1})\Gamma(m_2)\dots \Gamma(m_d)}{\Gamma \left( \sum_{i=1}^{d+1} m_i\right)}
	\end{align}
	which is finite by hypothesis.
	
	By an analogous reasoning, we obtain
		\begin{multline}
		I_{d+1}^{(d+1)}(m_1,\dots , m_{d+1}) \leqslant \beta (m_1,m_{d+1})\times I_d^{(d)}(m_2,\dots m_d, m_1+m_{d+1}) \\
		 + \dfrac{\Gamma (m_1+m_{d+1})\Gamma (m_2) \dots \Gamma (m_d)}{\Gamma \left( \sum_{i=1}^{d+1} m_i \right)}\times I_2^{(2)}(m_1,m_{d+1}) 
		\end{multline}
	which is finite, and it concludes the proof.
	\end{enumerate}	
\end{proof}

\begin{lem}\label{LogisticLip}
Suppose that $\nu_{\mu,s}$ is the multivariate logistic normal 
with parameters $\mu$ and covariance matrix $sV$ where $s>0$ and $V$ is a given symmetric positive definite matrix. Let also $a>0$ and $b>1$. There then exists $K>0$ such that 
			\begin{equation}
				\forall (\mu,s), (\mu',s') \in [-a,a]^{d-1}\times [1/b,b], \ d_{TV}\left( \nu_{\mu,s} ; \nu_{\mu',s'}\right) \leqslant K \left[\vert \mu - \mu'\vert+\vert s-s'\vert\right].
			\end{equation}
\end{lem}

\paragraph{Proof of Lemma \ref{LogisticLip}}
Without loss of generality, we assume that $\vert\cdot\vert$ is the Euclidean norm.
For a measurable mapping $h:\mathcal{S}_{d-1}\rightarrow [-1,1]$, we have 
$$\int h d\nu_{\mu,s}=\int h\circ \mbox{ alr }^{-1} \phi_{\mu,s}d\lambda_{d-1},$$
where $\phi_{\mu,s}$ is the Gaussian density with mean $\mu$ and variance $sV$
and $\lambda_{d-1}$ is the Lebesgue measure on $\R^{d-1}$.
We then get 
$$d_{TV}\left( \nu_{\mu,s} ; \nu_{\mu',s'}\right)\leq \frac{1}{2}\int\left\vert \phi_{\mu,s}-\phi_{\mu',s'}\right\vert d\lambda_{d-1}.$$
Using Proposition $2.1$ in \citet{devroye2018total}, we get the bound
\begin{equation}\label{nice}
d_{TV}\left( \nu_{\mu,s} ; \nu_{\mu',s'}\right)\leq \frac{1}{2}\sqrt{(d-1)(s'/s-1)+(\mu-\mu')'s^{-1}V^{-1}(\mu-\mu')'-(d-1)\log(s'/s)}.
\end{equation}
In the case $s'\geq s$, one can use the inequality 
$$x-1-\log(x)\leq (x-1)^2,\quad x\geq 1.$$
We then obtain
$$2d_{TV}\left( \nu_{\mu,s} ; \nu_{\mu',s'}\right)\leq \sqrt{d-1}\vert s'/s-1\vert+s^{-1/2}\lambda^{-1/2}\vert \mu-\mu'\vert,$$
where $\lambda>0$ is the smallest eigenvalue of $V$. Since, $s,s'$ are located in a compact interval included in the set of positive real numbers, the result is now straightforward. The case $s'<s$ is similar, by inverting the roles of $(\mu,s)$ and $(\mu',s')$ in (\ref{nice}).

\hfill \qedsymbol

\subsection{Proof of Proposition \ref{cube}.}

We apply Corollary \ref{corexist}.
We only check {\bf H1-H2} for the measure $\mu_{\alpha}$, the rest of the proof being similar to the proofs of Proposition \ref{PDirichlet1} or Proposition \ref{PDirichlet2}.
For $\alpha=(\alpha_1,\ldots,\alpha_{d+1})$ in $(0,\infty)^{d+1}$, the measure $\mu_{\alpha}$ has a density w.r.t. the Lebesgue measure on $\R^d$ given by 
$$f_{\alpha}(u_1,\ldots,u_d)=\frac{\Gamma(\alpha_1+\cdots+\alpha_{d+1})\prod_{i=1}^d \frac{u_i^{\alpha_i-1}}{(1-u_i)^{\alpha_i+1}}\mathds{1}_{(0,1)}(u_i)}{\Gamma(\alpha_1)\cdots \Gamma(\alpha_{d+1})\left(1+\sum_{i=1}^d \frac{u_i}{1-u_i}\right)^{\sum_{i=1}^d \alpha_i}}.$$
Since $f_{\alpha}$ is continuous and positive on its support, it is straightforward to show {\bf H1}. For $s>0$, let $g_s$ be the probability density of the gamma distribution with parameter $s$.
To check {\bf H2}, for $1\leq i\leq d+1$, let $Z_i$ and $Z_i'$ be two random random variables with a gamma distribution with respective parameters $\alpha_i$ and $\alpha_i'$. We suppose all the random variables independent. Let $h:\R^d\rightarrow [0,1]$ be a measurable function. If $U:=(U_1,\ldots,U_d)$, $U'=(U_1',\ldots,U_d')$ and $V=(V_1,\ldots,V_d)$ with $U_i=Z_i/(Z_i+Z_{d+1})$, $U_i'=Z_i'/(Z_i'+Z_{d+1}')$ and $V_i=Z_i/(Z_i+Z_{d+1}')$ for $1\leq i\leq d$, we have
$\E h(U)-\E h(U')=\E h(U)-\E h(V)+\E h(V)-\E h(U')$. From the independence properties and Fubini's theorem, we have
$$\left\vert \E h(U)-\E h(V)\right\vert\leq \int \vert g_{\alpha_{d+1}}(x)-g_{\alpha_{d+1}'}(x)\vert dx$$
and
\begin{eqnarray*}
\left\vert \E h(V)-\E h(U')\right\vert&\leq& d_{TV}\left(\P_{(Z_1,\ldots,Z_d)},\P_{(Z_1',\ldots,Z_d')}\right)\\
&\leq & \sum_{i=1}^d d_{TV}\left(\P_{Z_i},\P_{Z_i'}\right)\\
&\leq & \frac{1}{2}\sum_{i=1}^d\int \left\vert g_{\alpha_i}(x)-g_{\alpha_i'}(x)\right\vert dx.
\end{eqnarray*}
It is then possible to show that for any compact subset $H$ of $(0,\infty)$, there exists $C>0$ s.t. for all $s,s'\in H$,
$$\int \left\vert g_s(x)-g_{s'}(x)\right\vert dx\leq C\vert s-s'\vert.$$
This yields to the validity of {\bf H2} and concludes the proof.$\square$

\subsection{Proof of Proposition \ref{sphere2}}
Before proving the result, we state the following lemma, which ensures a contraction property of a rotation with respect to its axis.
\begin{lem}\label{ax}
For $u$ and $v$ on the unit sphere,
\begin{equation}\label{cont}
\Vert \mathcal{R}_{u,\theta}-\mathcal{R}_{v,\theta}\Vert_o\leq  \kappa_{\theta}\Vert u-v\Vert,
\end{equation}
with $\kappa_{\theta}=2\sqrt{2}\sqrt{1-\cos(\theta)}$.
Here $\Vert\cdot\Vert_o$ denotes the norm $\Vert A\Vert_o=\sqrt{\rho(A^TA)}$ for a square matrix $A$, i.e. the operator norm associated to the Euclidean norm.
\end{lem}

\paragraph{Proof of Lemma \ref{ax}.} Suppose first that $u$ and $v$ are not colinear. We introduce two orthonormal basis: $\{u,w,w_1\}$ and $\{v,w,w_2\}$ where $w$ is orthogonal to $u$ and $v$ 
and $w_1=\mathcal{R}_{w,\pi/2}u$, $w_2=\mathcal{R}_{w,\pi/2}v$. We have the decomposition 
$$\mathcal{R}_{u,\theta}=O_1 D(\theta)O_1^T,\quad \mathcal{R}_{u,\theta}=O_2 D(\theta)O_2^T,$$
where $O_1=[u,w,w_1]$, $O_2=[v,w,w_2]$ and 
$D(\theta)=\begin{pmatrix} 1&0&0\\ 0&\cos(\theta)&-\sin(\theta)\\0& \sin(\theta)&\cos(\theta)\end{pmatrix}$. This yields to 
$$\mathcal{R}_{u,\theta}-\mathcal{R}_{v,\theta}=\left(O_1-O_2\right)\left(D(\theta)-I_3\right)O_1^T+O_2\left(D(\theta)-I_3\right)\left(O_1-O_2\right)^T.$$
This yields to the inequality
$$\Vert \mathcal{R}_{u,\theta}-\mathcal{R}_{v,\theta}\Vert_o\leq 2\Vert O_1-O_2\Vert\cdot\Vert D(\theta)-I_3\Vert.$$
Since 
$$(O_1-O_2)^T(O_1-O_2)=\begin{pmatrix} \Vert u-v\Vert^2&0&0\\0&0&0\\0&0&\Vert u-v\Vert^2\end{pmatrix},$$
we have $\Vert O_1-O_2\Vert_o=\Vert u-v\Vert$. Moreover 
$$\left(D(\theta)-I_3\right)^T\left(D(\theta)-I_3\right)=\begin{pmatrix} 0&0&0\\0& 2-2\cos(\theta)&0\\0&0& 2-2\cos(\theta)\end{pmatrix}.$$
Then $\Vert D(\theta)-I_3\Vert_o\leq \sqrt{2-2\cos(\theta)}$.
We then get (\ref{cont}) with $\kappa_{\theta}=2\sqrt{2}\sqrt{1-\cos(\theta)}$.

Next, if $u$ and $v$ are colinear, then either $v=u$ and the result is immediate or $v=-u$ and in this case, $\mathcal{R}_{v,\theta}=\mathcal{R}_{u,-\theta}$. In the basis $\{u,w,w_1\}$ mentioned above, we have
$$\mathcal{R}_{u,\theta}-\mathcal{R}_{v,\theta}=O_1\left(D(\theta)-D(-\theta)\right)O_1^T,$$
which shows that 
$$\Vert \mathcal{R}_{u,\theta}-\mathcal{R}_{v,\theta}\Vert^2_o=4 \sin^2(\theta)\leq 32 \left(1-\cos(\theta)\right)=\kappa_{\theta}\Vert u-v\Vert.$$
This ends the proof of the proposition.
$\square$

We now go back to the proof of Proposition \ref{sphere}. To apply Corollary \ref{corexist} $1$, it suffices to check Assumption {\bf H3}. Assumptions {\bf H1-H2} were already discussed before the statement of the proposition.
 If $y$ and $z$ are two sequences on the unit sphere such that $y_i=z_i$ for $1\leq i\leq m$, we get from Lemma \ref{ax} 
\begin{eqnarray*}
&&\left\vert \mu(y)-\mu(z)\right\vert\\
&=&\lim_{k\rightarrow \infty}\Vert g_{1,y_1}\circ\cdots\circ g_{k,y_k}(s)- g_{1,y_1}\circ\cdots\circ g_{m,y_m}\circ g_{m+1,z_{m+1}}\circ\cdots\circ g_{k,z_k}(s)\Vert\\
&\leq& \lim_{k\rightarrow \infty}\prod_{i=1}^m \left(2\sqrt{2}\sqrt{1-\cos(\theta_j)}\right)\Vert g_{m+1,y_{m+1}}\circ\cdots\circ g_{k,y_k}(s)-g_{m+1,z_{m+1}}\circ\cdots\circ g_{k,z_k}(s)\Vert\\
&\leq& 2\prod_{i=1}^m \left(2\sqrt{2}\sqrt{1-\cos(\theta_j)}\right).
\end{eqnarray*}
Setting 
$$b_m=\sup\left\{\left\vert \mu(y)-\mu(z)\right\vert: y_i=z_i, 1\leq i\leq m\right\},$$
our assumption on the $\theta_j$'s implies that $\sum_{m=1}^{\infty}b_m<\infty$ and Assumption {\bf H3} of Corollary \ref{corexist} is satisfied.$\square$

\subsection{Proof of Proposition \ref{prop:4.8}}
\begin{enumerate}
\item
We first recall that if $(p_1,\ldots,p_d)$ and $(q_1,\ldots,q_d)$ are two probability measures on a finite set and with a positive probability mass function, the properties of the Kullback-Leibler divergence ensure that 
$$-\sum_{i=1}^dp_i\log q_i\leq -\sum_{i=1}^dp_i\log p_i$$
with an equality if and only if $q_i=p_i$ for $i=1,\ldots,d$. 

Next, we notice that
$$\E m_0\left(\widetilde{\theta}_1\right)=-\sum_{i=1}^d\E\left[\lambda_{i,0}(\theta_1)\log\lambda_{i,0}\left(\widetilde{\theta}_1\right)\right].$$
Applying the previous property, we directly get the result, since the conditional mean takes values in $\mathcal{S}_{d-1}$.
\item
Take any $t \in \left\lbrace 1, \dots , n \right\rbrace$, the application:
	\begin{equation}
	\widetilde{\theta}_1 \longmapsto \mu_t\left(\widetilde{\theta}_1\right)
	\end{equation}
is actually linear in this set up. Furthermore, straightforward computations lead to
\begin{eqnarray*}
&&\sum_{i=1}^dY_{t,i} \log (\lambda_{i,t}(\widetilde{\theta}_1))\\
&=&\sum_{i=1}^{d-1} Y_{i,t}\mu_{i,t}(\widetilde{\theta}_1)-\log \left( 1+ \sum_{j=1}^{d-1} \exp (\mu_{j,t}(\widetilde{\theta}_1))\right)\\ 
&:=& g_t\left(\mu_t(\widetilde{\theta}_1)\right)
\end{eqnarray*}
where $g_t(u)=\sum_{i=1}^{d-1}Y_{i,t} u_i - \log \left( 1+ \sum_{j=1}^{d-1} \exp (u_j)\right)$.

We will prove that the application $g_t$ is concave, which will give the desired result by linearity of $\mu_t$, then taking the opposite and the sum over $t$.

\medskip

By differentiating twice the application $g_t$, we obtain for all $u\in \R^{d-1}$, and all $i, j \in \left\lbrace 1, \dots , d-1 \right\rbrace$ with $i \neq j$ :
\begin{equation}
\dfrac{\partial^2 g_t}{\partial u_i \partial u_j}(u) = \dfrac{\exp(u_i+u_j)}{\left(1+\sum_{k=1}^{d-1}\exp(u_k)\right)^2} = v_i\cdot v_j
\end{equation}
and :
\begin{equation}
\dfrac{\partial^2 g_t}{\partial u_i^2}(u) = \dfrac{-\exp(u_i)-\sum_{k\neq i} \exp(u_i+u_k)}{\left( 1 + \sum_{k=1}^{d-1} \exp(u_k)\right)^2} = -v_i\left(1-u_i\right)
\end{equation}
where $v_k=\dfrac{\exp(u_k)}{1+\sum_{l=1}^{d-1} \exp (u_l)}$.

Let's call $M$ the hessian matrix of $g_t$, we have for any $h \in \R^{d-1}$ :
\begin{align}
h'\cdot M \cdot h & = \sum_{i,j} h_i M_{i,j}h_j \notag \\
 & = \sum_{i\neq j}h_ih_j M_{i,j} + \sum_{i=1}^{d-1} h_i^2M_{i,i}\notag \\
 & = \sum_{i\neq j} v_ih_iv_jh_j - \sum_{i=1}^{d-1} h_i^2(v_i(1-v_i))\notag \\
 & = \sum_{i\neq j} v_ih_iv_jh_j + \sum_{i=1}^{d-1} v_i^2h_i^2 - \sum_{i=1}^{d-1} v_i h_i^2\notag \\
 & = \left( \sum_{i=1}^{d-1} v_ih_i \right)^2 - \sum_{i=1}^{d-1} v_ih_i^2.
\end{align}
However, by convexity, we have :
\begin{equation}
\left( \dfrac{\sum_{i=1}^{d-1} v_ih_i}{\sum_{j=1}^{d-1} v_j}\right)^2 \leqslant \dfrac{\sum_{i=1}^{d-1} v_i h_i^2}{\sum_{j=1}^{d-1}v_j}
\end{equation}
which implies :
\begin{equation}
\left( \sum_{i=1}^{d-1} v_ih_i\right)^2 \leqslant \left( \sum_{i=1}^{d-1}v_i h_i^2 \right) \left( \sum_{j=1}^{d-1} v_j \right).
\end{equation}
Thus :
\begin{equation}
h'\cdot M \cdot h \leqslant \left( \sum_{i=1}^{d-1} v_ih_i^2 \right) \left( \sum_{j=1}^{d-1} v_j -1 \right) \leqslant 0.
\end{equation}
Hence the result.
\end{enumerate}

\bibliographystyle{plainnat}
\bibliography{references}
\end{document}